\documentclass[a4paper,11pt,pdf]{amsart}

\usepackage{enumerate, amsmath, amsfonts, amssymb, amsthm, thmtools, wasysym, graphics, graphicx, xcolor, frcursive,comment,bbm}

\usepackage{etex}

\usepackage{vmargin}            
\setmarginsrb{3.3cm}{3.2cm}{3.3cm}{2.8cm}{0cm}{0.6cm}{0cm}{0cm}

\usepackage{array}

\makeatletter
\newcommand{\thickhline}{%
    \noalign {\ifnum 0=`}\fi \hrule height 1pt
    \futurelet \reserved@a \@xhline
}
\newcolumntype{"}{@{\hskip\tabcolsep\vrule width 1pt\hskip\tabcolsep}}
\makeatother

\def\W{{\mathcal{W}}}
\def\inv{{\sf{inv}_c}}
\newcommand\bD{{\boldsymbol D}}

\def\NC{{\sf{NC}}}

\def\In{\mathcal{I}(n)}

\def\Pol{{\sf{Pol}}}

\def\A{{\mathbb{Z}[v, v^{-1}]}}

\newcommand\bm{{\boldsymbol m}}

\newcommand\tl{{\mathrm{TL}_n}}
\newcommand\Bni{{\mathcal{B}_{n+1}}}

\def\ncs{{x}}

\def\ls{{\ell_{\mathcal{S}}}}
\def\T{{\mathcal{T}}}
\def\lt{{\ell_{\T}}}
\def\leqt{{\leq_{\T}}}

\def\leqv{\leq_{\mathcal{V}}}
\def\Std{{\sf{Std}}}

\def\cl{c_{\mathrm{lin}}}
\def\F{{\sf{FC}}}
\def\S{{\mathcal{S}}}

\usepackage{etex}

\definecolor{darkblue}{rgb}{0.0,0,0.7} 
\newcommand{\darkblue}{\color{darkblue}} 
\definecolor{darkred}{rgb}{0.7,0,0} 
\newcommand{\defn}[1]{\emph{\darkblue #1}} 

\usepackage[all]{xy}
\usepackage[T1]{fontenc}

\usepackage[all]{xy}
\usepackage[T1]{fontenc}

\usepackage[colorinlistoftodos]{todonotes}

\def\NC{{\sf{NC}}}

\def\Pol{{\sf{Pol}}}

\def\ls{{\ell_{\mathcal{S}}}}

\def\Vert{{\sf{Vert}}}

\def\Sn{{\mathfrak{S}_{n+1}}}
\def\supp{{\sf{supp}}}
\def\supps{{{\sf{supp}}}_{\mathcal{S}}}
\def\Subf{{\sf{Sub}_{\sf{FC}}}}

\newtheorem{theorem}{Theorem}[section]
\newtheorem{proposition}[theorem]{Proposition}
\newtheorem{lemma}[theorem]{Lemma}
\newtheorem{corollary}[theorem]{Corollary}

\theoremstyle{definition}
\newtheorem{definition}[theorem]{Definition}
\newtheorem{rmq}[theorem]{Remark}
\newtheorem{exple}[theorem]{Example}
\newtheorem{nota}[theorem]{Notation}

\theoremstyle{remark}

\usepackage{graphicx}                  
\usepackage{pstricks,pst-plot,pst-text,pst-tree,pst-eps,pst-fill,pst-node,pst-math}
\usepackage{setspace}

\newcommand{\ts}{\textsuperscript}

\title{Coxeter-Catalan combinatorics and Temperley-Lieb algebras}

\author{Thomas Gobet}
\address{TU Kaiserslautern, Fachbereich Mathematik, Postfach 3049, 67653 Kaiserslautern, Germany.}
\email{gobet@mathematik.uni-kl.de} 

\begin{document}
\maketitle
\begin{abstract}
We introduce bijections between generalized type $A_n$ noncrossing partitions (that is, associated to arbitrary standard Coxeter elements) and fully commutative elements of the same type. The latter index the diagram basis of the classical Temperley-Lieb algebra, while for each choice of standard Coxeter element the corresponding noncrossing partitions also index a basis, given by the images in the Temperley-Lieb algebra of the simple elements of the dual Garside structure (associated to this choice of standard Coxeter element) of the Artin braid group on $n+1$ strands. We then show that our bijections come from triangular base changes between the diagram basis and the various bases indexed by noncrossing partitions, by explicitly describing the orders giving triangularity. These orders were introduced in a joint paper with Williams and provide exotic lattice structures on noncrossing partitions. Several combinatorial objects are introduced along the way, including an involution on the set of noncrossing partitions.\\
~\\
  
\bigskip


\end{abstract}
\tableofcontents

\section{Introduction}

Catalan numbers appear in many locations in mathematics (see \cite{Stanley} for a list of objects counted by the Catalan numbers). One famous set of combinatorial objects having Catalan enumeration is the set of noncrossing partitions \cite{kreweras1972partitions}. Noncrossing partitions can naturally be seen as elements of the symmetric group $\mathfrak{S}_{n+1}$ associated to a fixed $(n+1)$-cycle, allowing a generalization of noncrossing partitions to finite Coxeter groups by replacing the $(n+1)$-cycle by a Coxeter element. The refinement order endows noncrossing partitions with a lattice structure satisfying many interesting properties. We refer the reader to \cite{Arm} and the references thereof for more on the topic.  

The aim of this paper is to explore a link between noncrossing partitions and (classical) Temperley-Lieb algebras. To any integer $n\geq 1$ is attached a Temperley-Lieb algebra $\tl$. It is an associative, unital algebra over a ring $\mathcal{A}$ of Laurent polynomials, which is free as an $\mathcal{A}$-module, of rank equal to the Catalan number $C_{n+1}=\frac{1}{n+2}\binom{2(n+1)}{n+1}$ (see for instance \cite{KaTu}). It has a basis indexed by planar diagrams or fully commutative permutations of the symmetric group $\mathfrak{S}_{n+1}$. On the other hand, the set of noncrossing partitions of $\mathfrak{S}_{n+1}$ has also cardinality $C_{n+1}$. 

Zinno \cite{Z} discovered an alternative basis of $\tl$ indexed by so-called \textit{canfacs}. Canfacs are certain elements of the $(n+1)$-strand Artin braid group, which turn out to be an incarnation of noncrossing partitions of the cycle $(1~2~\cdots~n+1)$. There is a natural multiplicative homomorphism between the $(n+1)$-strand Artin braid group and the Temperley-Lieb algebra $\tl$ and Zinno's basis is obtained by mapping the canfacs from the Artin braid group to the Temperley-Lieb algebra. 

The canfacs are of particular interest for the study of the braid group: they form a distinguished set of so-called \textit{simple} elements of a Garside structure on the Artin-braid group, called the \textit{dual} Garside structure (we do not elaborate here on Garside structures and information provided by such structures since it is not relevant for our purposes; the interested reader is encouraged to have a look at \cite{DDGKM} for more on Garside monoids and groups). It was originally discovered by Birman-Ko-Lee \cite{BKL}.

Birman-Ko-Lee's approach was later generalized to noncrossing partitions of an arbitrary standard Coxeter element of a finite Coxeter group (\cite{BDM}, \cite{BW1}, \cite{Dual}). More precisely, given a finite Coxeter group and a standard Coxeter element one can associate to this data a Garside structure on the corresponding Artin-Tits group and the simple elements of the Garside structure are lifts of the noncrossing partitions of the chosen standard Coxeter element in the Artin-Tits group. The dual Garside structure on an Artin-Tits group is therefore not unique: it depends on a choice of standard Coxeter element. Zinno's canfacs correspond to the case where the Coxeter element is the $(n+1)$-cycle $(1~2~3~\cdots~n+1)$. It is natural to ask whether one always gets a basis of $\tl$ when mapping the simple elements of a dual Garside structure of the $(n+1)$-strand Artin braid group to the Temperley-Lieb algebra. This holds as shown by Vincenti \cite{Vincenti} (following an idea of Lee and Lee \cite{LL}). 

In this paper, we explore further the relation between the above mentioned family of bases of $\tl$ (there is one basis for each Coxeter element and the basis is indexed by the noncrossing partitions of that Coxeter element) and the basis indexed by fully commutative permutations (the "diagram basis"). More precisely, we show that there is always a triangular base change between any of the bases indexed by noncrossing partitions and the diagram basis, refining results of \cite{Z} and \cite{Gob} (there the Coxeter element is $c=(1~2~\cdots~n+1)$). We explicitly give the orderings of the bases yielding triangularity and give a full combinatorial definition and description of the bijections between noncrossing partitions and fully commutative elements induced by these base changes. At the level of noncrossing partitions, these orders are linear extensions of an exotic lattice structure on noncrossing partitions introduced in a joint work with Williams \cite{GobWil} (which is isomorphic to the lattice of nonnesting partitions). Our bijections make use of a new involution on the set of noncrossing partitions. 

The paper is organized as follows: in Section $2$ we recall the definition and properties of noncrossing partitions. In Section $3$ we introduce bijections between noncrossing partitions and fully commutative elements. In Section $4$ we recall some results on the exotic lattice structure on noncrossing partitions mentioned above. In Section $5$ we describe the incarnations of noncrossing partitions in the Artin braid group. In Section $6$ we recall the definition and some properties of the classical Temperley-Lieb algebra. In Section $7$ we show the triangular base change between the diagram basis and the basis obtained by mapping the lifts of noncrossing partitions in the Artin braid group from Section $5$ to the Temperley-Lieb algebra. We also show that the triangular base change is induced by linear extensions of the orders from Section $4$; the way of ordering the diagram basis (indexed by fully commutative permutations) is given by taking the image of these linear extensions under the bijections from Section $3$ (in particular, the bijections induced from these base changes are those from Section $3$).

\section{Noncrossing partitions}

The aim of this section is to introduce noncrossing partitions both in a combinatorial and algebraic way.

\subsection{Combinatorial noncrossing partitions}\label{sec:combin}
For $n\in\mathbb{Z}_{\geq 0}$ we denote by $[n]$ the set $\{1,2,\dots, n\}$. Let $n\in\mathbb{Z}_{\geq 1}$.  A \defn{set partition} of $[n+1]$ is a collection $x$ of nonempty disjoint subsets of $[n+1]$ whose union is all of $[n+1]$; these subsets are called the \defn{blocks} of the set partition.  Given a set partition $x$, a \defn{bump} is a pair $(i,k)$ with $i<k$ in the same block of $x$ such that there is no $j$ in that block with $i<j<k$.

A \defn{noncrossing partition} $x$ of $[n+1]$ is a set partition with the condition that
		if $(i_1,i_2)$, $(j_1,j_2)$ are two distinct bumps in $x$, then it is not the case that $i_1<j_1<i_2<j_2$. Noncrossing partitions of $[n+1]$ are known to be counted by the Catalan number $C_{n+1}:=\frac{1}{n+2}\binom{2(n+1)}{n+1}$. They are item $159$ in \cite{Stanley}. 

The \defn{graphical representation} of a noncrossing partition $\ncs$ is the set of convex hulls of the blocks of $\ncs$ when drawn around a regular $(n+1)$-gon with vertices labeled with $[n+1]$ in clockwise order. The definition above is equivalent to the non-intersection of the hulls. An example is given in Figure \ref{figure:nonc}. Noncrossing partitions are item $160$ in \cite{Stanley}. Ordering them by refinement yields the \defn{noncrossing partition lattice}~\cite{kreweras1972partitions}. 

We write $\Pol(\ncs)$ for the set of polygons occurring in the graphical representation of a noncrossing partition $\ncs$ (an edge is always seen as a polygon, but a single point not). A polygon $P\in\Pol(\ncs)$ is given by a sequence of indices $i_1, i_2, \dots ,i_k,$ where the $i_j$'s index the vertices of $P$ and $i_1<i_2<\dots<i_k$. We set $\Vert(P):=\{i_1, i_2, \dots, i_k \}$ and say that the index $i_1$ of $P$ is \defn{initial} and $i_k$ is \defn{terminal}. For example, in Figure~\ref{figure:nonc}, the initial index of $P_1$ is $2$ and the terminal one is $5$, while the initial index of $P_2$ is $1$ and the terminal one is $6$. For convenience we will sometimes write $P=[i_1 i_2\cdots i_k]$.


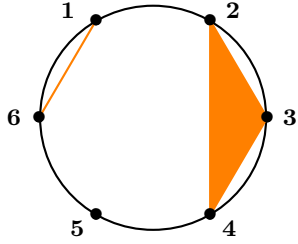
\begin{figure}[htbp]
\begin{center}
\begin{tabular}{ccc}
& \begin{pspicture}(0,0)(6,3)
\pscircle(3,1.5){1.5}
\psdots(4.5,1.5)(3.75,2.79)(3.75,.21)(2.25,.21)(2.25,2.79)(1.5,1.5)
\pspolygon[linecolor=orange, fillstyle=solid, fillcolor=orange](3.75,2.79)(4.5,1.5)(3.75,.21)
\psline[linecolor=orange](2.25,2.79)(1.5,1.5)
\rput(3.75,.21){$\bullet$}
\rput(4.5,1.5){$\bullet$}
\rput(3.75,2.79){$\bullet$}
\rput(2.25,.21){$\bullet$}
\rput(2.25,2.79){$\bullet$}
\rput(1.5,1.5){$\bullet$}
\rput(4.05,2.92){\textrm{{\footnotesize \textbf{2}}}}
\rput(4.8,1.5){\textrm{{\footnotesize \textbf{3}}}}
\rput(4,.03){\textrm{{\footnotesize \textbf{4}}}}
\rput(2,.03){\textrm{{\footnotesize \textbf{5}}}}
\rput(1.17,1.5){\textrm{{\footnotesize \textbf{6}}}}
\rput(1.88,2.92){\textrm{{\footnotesize \textbf{1}}}}
\end{pspicture} & \\
\end{tabular}
\end{center}
\caption{Graphical representation of the noncrossing partition $\ncs=\{\{1,6\},\{2,3,4\}, \{5\}\}$. There are two polygons $P_1=[234]$ and $P_2=[16]$.}
\label{figure:nonc}
\end{figure}



\subsection{Coxeter elements and algebraic noncrossing partitions}\label{sec:nc_algebraic}
Let $(\mathcal{W}, \mathcal{S})$ be a Coxeter system with $\S$ finite. This is to say that $\W$ admits a presentation with generators $\S$ of order two and relations (called \defn{braid relations}) of the form $sts\cdots = tst\cdots$, $s, t\in S$, $s\neq t$, where the number $m_{st}\in \{2,3,\dots\}\cup \{\infty\}$ of letters in the left hand side equals the number $m_{ts}$ in the right hand side (see \cite{bourbaki} or \cite{humph} for basics on Coxeter groups). Elements of $\S$ are called \defn{simple reflections}. The \defn{length} of $w \in \W$ is the minimal number $\ls(w)$ of elements of $\S$ required to write $w = s_{i_1}\cdots s_{i_{\ls(w)}}$ and such a product is an \defn{$\S$-reduced expression} for $w$. A word in the elements of $\S$ is called an \defn{$\S$-word}. We call \defn{$\S$-support} and denote by $\supps(w)$ the subset of $\S$ consisting of those simple reflections which occur in a reduced expression of $w$. It is well-defined since one can pass from any reduced expression of $w$ to any other by applying a sequence of braid relations.  

The set $\mathcal{T}:=\bigcup_{w\in\mathcal{W}} w\mathcal{S}w^{-1}$ is the set of \defn{reflections} of $\mathcal{W}$. The \defn{reflection length} of an element $w \in \W$ is the minimal number $\lt(w)$ of reflections required to write $w = s_{i_1}\cdots s_{i_{\lt(w)}}$ and such a product is a \defn{$\T$-reduced expression} for $w$. A word in the elements of $\T$ is called a \defn{$\T$-word}. The reflection length induces the \defn{absolute order} $\leqt$ on $\W$, that is, \[ u\leqt v \text{ if and only if } \lt(u)+\lt(u^{-1} v)=\lt(v).\] Notice that conjugation preserves $\leqt$.

A \defn{standard Coxeter element} in $\mathcal{W}$ is a product of all the elements of $\S$ in some order. We denote by $\Std(\mathcal{W})$ the set of standard Coxeter elements. 


For $c\in\Std(\mathcal{W})$ we let \[\NC(\mathcal{W},c):=\{u\in \W~|~u\leqt c \}\] be the order ideal of the Coxeter element $c$ (endowed with the order induced by $\leqt$). 

Let us from now on assume in addition that $\W$ is finite. One then has $\T\subseteq\NC(\mathcal{W},c)$ for any $c\in\Std(\mathcal{W})$ (see \cite[Lemma 1.3.3]{Dual}). For irreducible $\W$ the set $\NC(\W,c)$ is counted by the generalized Catalan number $\mathrm{Cat}(\W)$ (see \cite[Section 5.2]{Dual}). It is known to be a lattice (see \cite{Dual}, \cite{BW}) usually called \defn{noncrossing partition lattice} -- the link with the noncrossing partition lattice defined in Subsection \ref{sec:combin} is explained in Subsection \ref{sub:symm} below. Distinct Coxeter elements yield isomorphic lattices. For more on these lattices we refer to \cite{Arm} and the references thereof. 

\subsection{The case of the symmetric group}\label{sub:symm} The combinatorial noncrossing partitions of Subsection \ref{sec:combin} correspond to the algebraic noncrossing partitions of Subsection \ref{sec:nc_algebraic} in the case where $\W$ is the symmetric group and $c$ is as follows: using the identifications $\W=\Sn$, $\S=\{s_i:=(i,i+1)\}_{i=1}^n$, the Coxeter element is the $(n+1)$-cycle $c_{\mathrm{lin}}:=(1,2,\dots,n+1)=s_1 s_2\dots s_n$ which we call the \defn{linear} Coxeter element. In this paper we will only be interested by noncrossing partitions associated to standard Coxeter elements in the symmetric group, hence we will identify the elements of $\W$ group with permutations and the reflections with transpositions, and rather use the symmetric group terminology than the Coxeter theoretic one. 

Given $w\in\Sn$, define its \defn{support} ${\sf{supp}}(w)$ to be the complement of its set of fixed points in $\{1,\dots, n+1\}$. If $w$ is a cycle, then $\lt(w)=|{\sf{supp}}(w)|-1$ (see \cite[Proposition 2.3]{Brady}). For all $w\in \Sn$, the reflection length of $w$ is the sum of the reflection lengths of the cycles occurring in the decomposition of $w$ into a product of disjoint cycles (see \cite[Lemma 2.2]{Brady}). 

Each standard Coxeter element is an $(n+1)$-cycle but the converse is false in general; in the next subsection we explain how to characterize those $(n+1)$-cycles which lie in $\Std(\Sn)$. Mapping a permutation $\ncs \in \NC(\Sn,c_{\mathrm{lin}})$ to the set partition whose blocks are the supports of the cycles of $\ncs$ is an isomorphism of posets between $\NC(\Sn,c_{\mathrm{lin}})$ and the combinatorial noncrossing partition lattice (see \cite{Biane}). We explain in the next subsection how to generalize this to the case where $c$ is an arbitrary standard Coxeter element. 

\subsection{Graphical representations and conventions}\label{sec:conv}

Let $c\in\Std(\Sn)$. As in the case where $c=\cl$, the elements of $\NC(\Sn,c)$ can be represented by disjoint unions of polygons with vertices on a circle labeled by $\{1,\dots,n+1\}$. For $c=(i_1, i_2,\dots, i_{n+1})$, the labeling of the circle is given (in clockwise order) by $i_1 i_2\cdots i_{n+1}$. We call it the \defn{c-labeling}. Assuming $i_1=1$, $i_k=n+1$, one has that $c$ lies in $\Std(\Sn)$ if and only if the sequence $i_1i_2\cdots i_k$ is increasing while the sequence $i_k\cdots i_{n} i_{n+1}$ is decreasing (see \cite[Lemma 8.2]{GobWil}). An example is given in Figure \ref{figure:orientation}. To any cycle in the decomposition of $x\in\NC(\Sn,c)$ into a product of disjoint cycles, one associates the polygon obtained as convex hull of the set of points lying in the support of the cycle (we identify the points with their labels). The diagram obtained has the property that all the polygons are disjoint, equivalently that the partition defined by the cycle decomposition of $x$ is noncrossing for the $c$-labeling.


Indeed, every element in $\Std(\Sn)$ is a conjugate of $c_{\mathrm{lin}}$ in $\Sn$ (this is clear if one keeps in mind that standard Coxeter elements are $(n+1)$-cycles, but it is a general result for Coxeter groups whose Dynkin diagram is a tree that any two standard Coxeter elements are conjugate by a sequence of cyclic conjugations, see  \cite[Theorem 3.1.4]{GP}; in that case standard Coxeter elements are in bijection with orientations of the Dynkin diagram, see \cite[Theorem 1.5]{Shi}). Explicitly, if $c=(i_1,\dots, i_{n+1})\in\Std(\Sn)$ with $i_1=1$, define $w\in \Sn$ by $w(j)=i_j$ for each $j=1,\dots,n+1$. Then $c=wc_{\mathrm{lin}}w^{-1}$. Since conjugation preserves the partial order $\leqt$, we have $u\leqt c_{\mathrm{lin}}$ if and only if $wuw^{-1}\leqt c$, hence $v\leqt c$ if and only if the graphical representation of $v$ on the circle with $c$-labeling is noncrossing. Hence an element of $\NC(\Sn, c)$ will be called a \defn{$c$-noncrossing partition} or simply a \defn{noncrossing partition} if no confusion is possible. 

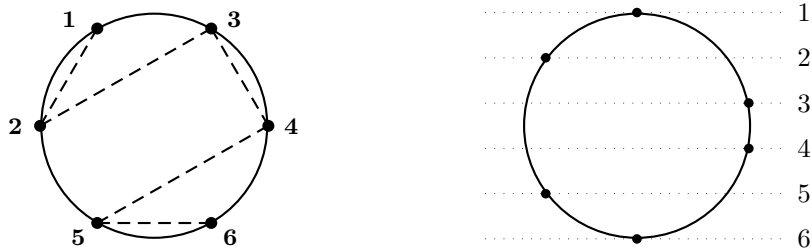
\begin{figure}[h!]
\begin{center}
\begin{tabular}{ccc}
& \begin{pspicture}(1,0)(6,3)
\pscircle(3,1.5){1.5}
\psdots(4.5,1.5)(3.75,2.79)(3.75,.21)(2.25,.21)(2.25,2.79)(1.5,1.5)
\psline[linestyle=dashed](2.25,2.79)(1.5,1.5)
\psline[linestyle=dashed](1.5,1.5)(3.75,2.79)
\psline[linestyle=dashed](3.75,2.79)(4.5,1.5)
\psline[linestyle=dashed](4.5,1.5)(2.25,.21)
\psline[linestyle=dashed](2.25,.21)(3.75,.21)
\psdots[dotsize=4.5pt](3.75,.21)(2.25,.21)
\psdots[dotsize=4.5pt](1.5,1.5)(2.25,2.79)(4.5,1.5)(3.75,2.79)
\rput(4.05,2.92){\textrm{{\footnotesize \textbf{3}}}}
\rput(4.8,1.5){\textrm{{\footnotesize \textbf{4}}}}
\rput(4,.03){\textrm{{\footnotesize \textbf{6}}}}
\rput(2,.03){\textrm{{\footnotesize \textbf{5}}}}
\rput(1.17,1.5){\textrm{{\footnotesize \textbf{2}}}}
\rput(1.88,2.92){\textrm{{\footnotesize \textbf{1}}}}
\end{pspicture} & \begin{pspicture}(0,0)(6,3)
\pscircle(3,1.5){1.5}

\psline[linestyle=dotted, linewidth=0.4pt](1,3)(5,3)
\psline[linestyle=dotted, linewidth=0.4pt](1,2.4)(5,2.4)
\psline[linestyle=dotted, linewidth=0.4pt](1,1.8)(5,1.8)
\psline[linestyle=dotted, linewidth=0.4pt](1,1.2)(5,1.2)
\psline[linestyle=dotted, linewidth=0.4pt](1,0.6)(5,0.6)
\psline[linestyle=dotted, linewidth=0.4pt](1,0)(5,0)

\psdots(3,3)(1.8, 2.4)(4.4696, 1.8)(4.4696, 1.2)(1.8, 0.6)(3,0)
\rput(5.2,3){\small $1$}
\rput(5.2,2.4){\small $2$}
\rput(5.2,1.8){\small $3$}
\rput(5.2,1.2){\small $4$}
\rput(5.2,0.6){\small $5$}
\rput(5.2,0){\small $6$}
\end{pspicture}
\end{tabular}
\end{center}
\caption{The $c$-labeling for $c=s_2 s_1 s_3 s_5 s_4=(1,3,4,6,5,2)\in\Std(\Sn)$. Note that when going from point $1$ to point $6$ we obtain a noncrossing zigzag.}
\label{figure:orientation}
\end{figure}

\begin{definition}\label{def:admiss}
Let $i,j,k\in[n+1]$ with $|\{i,j,k\}|=3$ and assume that $x=(i,j,k)\in\NC(\Sn,c)$. The triangle with vertices $i,j,k$ on the $c$-labeling represents the element $x\in\NC(\Sn, c)$ and we have $\ell_\T(x)=2$. We say that the ordered triple $(t_1=(i,j), t_2=(j,k), t_3=(k,i))$ (or any cyclic permutation of it) of reflections is a \defn{$c$-admissible triangle}. In that case $x$ is equal to the product of any two successive reflections in the triple, that is, $$x=t_1 t_2= t_2 t_3= t_3 t_1.$$
\end{definition}

It is more convenient for some proofs to slightly modify the graphical representation as follows: let $c=(i_1,\dots, i_{n+1})\in\Std(\Sn)$ with $i_1=1$, $i_k=n+1$. Instead of drawing the points on the circle such that the length of the segment between any two successive points is constant, we draw the point with label $1$ at the top of the circle, the point with label $n+1$ at the bottom, the points with label in $i_k\cdots i_n i_{n+1}$ on the left and the points with label in $i_1i_2\cdots i_k$ on the right, each point having a specific height depending on its label. That is, if $i, j\in[n+1]$, $i<j$, then the point $i$ is higher than the point $j$ (as done in the right of Figure \ref{figure:orientation}). Also, when representing a noncrossing partition we will use curvilinear polygons instead or regular polygons for convenience (as in Figure \ref{figure:deux}). 

In the above setting, we denote by $R_c$ (respectively $L_c$) the set $\{i_1,\dots, i_k\}$ (respectively $\{i_k,\dots, i_{n+1}, i_1\}$). In particular $L_{c}\cup R_c=[n+1]$, $L_{c}\cap R_c=\{1, n+1\}$. Notice that $R_{c_{\mathrm{lin}}}=[n+1]$, $L_{c_{\mathrm{lin}}}=\{1,n+1\}$. 

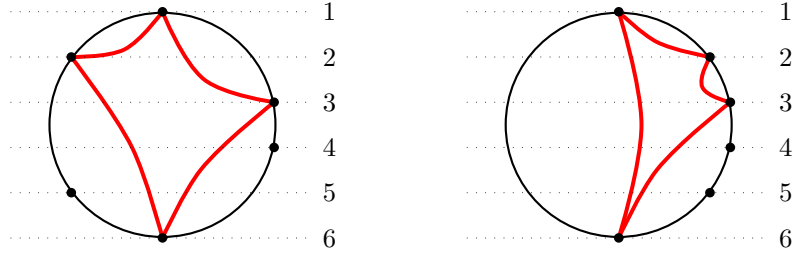
\begin{figure}[h!]
\psscalebox{1}{
\begin{pspicture}(6.3,0)(6,3)
\pscircle(3,1.5){1.5}
\pscurve[linecolor=red, linewidth=1.5pt](3,3)(2.5, 2.5)(1.8, 2.4)
\pscurve[linecolor=red, linewidth=1.5pt](3,3)(3.55, 2.1)(4.4696, 1.8)
\pscurve[linecolor=red, linewidth=1.5pt](3,0)(2.6,1.2)(1.8, 2.4)
\pscurve[linecolor=red, linewidth=1.5pt](3,0)(3.5,0.9)(4.4696, 1.8)

\psline[linestyle=dotted, linewidth=0.4pt](1,3)(5,3)
\psline[linestyle=dotted, linewidth=0.4pt](1,2.4)(5,2.4)
\psline[linestyle=dotted, linewidth=0.4pt](1,1.8)(5,1.8)
\psline[linestyle=dotted, linewidth=0.4pt](1,1.2)(5,1.2)
\psline[linestyle=dotted, linewidth=0.4pt](1,0.6)(5,0.6)
\psline[linestyle=dotted, linewidth=0.4pt](1,0)(5,0)

\psdots(3,3)(1.8, 2.4)(4.4696, 1.8)(4.4696, 1.2)(1.8, 0.6)(3,0)
\rput(5.2,3){\small $1$}
\rput(5.2,2.4){\small $2$}
\rput(5.2,1.8){\small $3$}
\rput(5.2,1.2){\small $4$}
\rput(5.2,0.6){\small $5$}
\rput(5.2,0){\small $6$}
\pscircle(9,1.5){1.5}
\pscurve[linecolor=red, linewidth=1.5pt](9,3)(9.5, 2.6)(10.2, 2.4)
\pscurve[linecolor=red, linewidth=1.5pt](10.2, 2.4)(10.1,2)(10.4696, 1.8)
\pscurve[linecolor=red, linewidth=1.5pt](10.4696, 1.8)(9.5,0.9)(9,0)
\pscurve[linecolor=red, linewidth=1.5pt](9,3)(9.3,1.5)(9,0)

\psline[linestyle=dotted, linewidth=0.4pt](7,3)(11,3)
\psline[linestyle=dotted, linewidth=0.4pt](7,2.4)(11,2.4)
\psline[linestyle=dotted, linewidth=0.4pt](7,1.8)(11,1.8)
\psline[linestyle=dotted, linewidth=0.4pt](7,1.2)(11,1.2)
\psline[linestyle=dotted, linewidth=0.4pt](7,0.6)(11,0.6)
\psline[linestyle=dotted, linewidth=0.4pt](7,0)(11,0)

\rput(11.2,3){\small $1$}
\rput(11.2,2.4){\small $2$}
\rput(11.2,1.8){\small $3$}
\rput(11.2,1.2){\small $4$}
\rput(11.2,0.6){\small $5$}
\rput(11.2,0){\small $6$}

\psdots(9,3)(10.2, 2.4)(10.4696, 1.8)(10.4696, 1.2)(10.2, 0.6)(9,0)\end{pspicture}}
\caption{Conventions on the graphical representations of $c$-noncrossing partitions. On the left we have $(1,3,6,2)\in\NC(\mathfrak{S}_6, s_2s_1s_3s_5s_4)$, on the right we have $(1,2,3,6)\in\NC(\mathfrak{S}_6,c_{\mathrm{lin}})$.}
\label{figure:deux} 
\end{figure}

\begin{rmq}\label{rmq:successifs}
For any $k\in\{1,\dots, n+1\}$ we denote by $\mathcal{L}_k$ the horizontal line containing the point $k$ and cutting the circle with $c$-labeling for $c\in\Std(\Sn)$. The set of points on the circle with labels in $E_{\leq k}:=\{i\in[n+1]~|~i\leq k\}$ lies in the upper half-plane defined by that line and the set of points with label in $E_{\geq k}:=\{i\in[n+1]~|~i\geq k\}$ lies in the lower half-plane. This will be useful for some of the proofs. 
\end{rmq}

\noindent To summarize, there are bijections 
$$\left\{\begin{array}{cccccc}\text{Standard}\\ \text{Coxeter}\\ \text{elements}\end{array}\right\} \overset{\sim}{\longrightarrow}  \left\{\begin{array}{cccccc} \text{Orientations}\\ \text{of the}\\ \text{Dynkin diagram} \end{array}\right\}
 \overset{\sim}{\longrightarrow}  \left\{\begin{array}{cccccc} \text{Partitions of }\\
 \text{$\{2,\dots,n\}$ into two}\\ \text{ordered sets $(C_1, C_2)$}  
\end{array}\right\}
$$
where by convention $C_1$ and $C_2$ are defined by $L_{c}=C_1\overset{\cdot}{\cup}\{1,n+1\}$, $R_{c}=C_2\overset{\cdot}{\cup} \{1,n+1\}$. 

Note that the first bijection above exists more generally for Coxeter groups whose Dynkin diagram is a tree (see \cite[Theorem 1.5]{Shi}). 

Let $c\in\Std(\Sn)$, $x\in\NC(\Sn,c)$. As in \ref{sec:combin}, we denote by $\Pol(x)$ the set of polygons appearing in the graphical representation of $x$. We denote by $\Vert(P)$ the set of integers indexing the vertices of $P\in\Pol(x)$. Notice that $\supp(x)=\bigcup_{P\in\Pol(x)} \Vert(P)$.

Let $P\in\Pol(x)$ with $\Vert(P)=\{d_1< d_2 <\cdots< d_k\}$. We say that $d_1$ is the \defn{initial} index of $P$ or an initial index of $x$. We say that $d_k$ is the \defn{terminal} index of $P$ or a terminal index of $x$. If we do not want to write down the whole set $\Vert(P)$ we simply denote by $\min P$ the initial index of $P$ and by $\max P$ the terminal one. We say that $\ell\in\{2,\dots, n\}$ is \defn{nested} in $P\in\Pol(x)$ if $\ell\notin \Vert(P)$ but $\min P<\ell<\max P$. Notice that it does not imply that $\ell\notin\supp(x)$ since one may have $\ell\in \Vert(Q)$ for $P\neq Q\in\Pol(x)$. As in \cite[Definition 8.8]{GobWil}, we set 
$$U_x^c:=\supp(x)\setminus \{\mbox{initial vertices}\},~ D_x^c:=\supp(x)\setminus \{\mbox{terminal vertices}\}.$$
We have $|D_x^{c}|=|U_x^{c}|$.

\begin{exple}
Let $x$ be as in Figure \ref{figure:nonc}. Then the point $5$ is nested in $P_2$, but not in $P_1$. We have $D_x^c=\{1,2,3\}$, $U_x^c=\{3,4,6\}$. 
\end{exple}


\section{Bijections with other Catalan enumerated objects}

In this section, we define two bijections between $\NC(\Sn,c)$ and fully commutative permutations of the symmetric group (see below for definitions). One of the introduced bijections will turn out to come from a triangular base-change in the Temperley-Lieb algebra (the sets of noncrossing partitions and fully commutative elements turn out to index bases of this algebra). We nevertheless treat everything combinatorially in this section, independently of the Temperley-Lieb algebra.

Firstly, we introduce a set of pairs of integral sequences which also has Catalan enumeration. Our bijections will then be built via intermediate bijections between the two above mentioned sets and this set of sequences.
\subsection{Pairs of integral sequences}

\begin{nota}
Let $k\in\mathbb{Z}_{\geq 0}$. We denote by $\mathcal{I}_k$ the set of pairs $(D, U)$ of subsets of $[n+1]$ where $D=\{d_1, d_2,\dots, d_k\}$, $U=\{u_1, u_2,\dots, u_k\}$, $d_i<d_{i+1}$, $u_i<u_{i+1}$ for each $1\leq i<k$ and $d_i< u_i$ for each $1\leq i\leq k$. Set $\mathcal{I}(n):=\coprod_{k=0}^n \mathcal{I}_k$. Notice that $\mathcal{I}(n)$ is item $107$ in \cite{Stanley}. In particular it is known to have enumeration $C_{n+1}$.
\end{nota}
\noindent Note that for any $c\in\Std(\Sn)$ and $x\in\NC(\Sn,c)$, we have $(D_x^c, U_x^c)\in\In$. 
   
\subsection{Bijection $\#$1}\label{sub:bij1}
In \cite[Proposition 2.11]{Gob} and \cite[Proposition 5.2]{GobWil}, it is shown that the map $$\varepsilon:\NC(\Sn, c_{\mathrm{lin}}){\longrightarrow}\mathcal{I}(n),~x\mapsto (D_x^{c_{\mathrm{lin}}}, U_x^{c_{\mathrm{lin}}})$$ is a bijection.

The next result is an immediate consequence of \cite{GobWil}. It is a generalization of the bijection $\varepsilon$ to arbitrary Coxeter elements. 
\begin{proposition}\label{prop:gobwilbij}
The map $\psi_1:\NC(\Sn, c)\longrightarrow \mathcal{I}(n)$, $x\mapsto (D_x^c, U_x^c)$ is well-defined and bijective. 
\end{proposition}
\begin{proof}
By \cite[Theorem 1.3]{GobWil}, there exists a unique bijection $$f:\NC(\Sn,c)\longrightarrow \NC(\Sn,c_{\mathrm{lin}})$$
such that $(D_{f(x)}^{c_{\mathrm{lin}}}, U_{f(x)}^{c_{\mathrm{lin}}})=(D_x^c, U_x^c)$. We get the result by combining this with the fact that the map $\varepsilon$ above is bijective. 
\end{proof}

\begin{rmq}
We mention another way of using the bijection $\varepsilon$ to produce bijections $\NC(\Sn,c)\longrightarrow \mathcal{I}(n)$. Since for any $c,c'\in\Std(\Sn)$ there are bijections $$\NC(\Sn,c)\overset{\sim}{\longrightarrow}\NC(\Sn,c')$$ given by conjugation (see the second paragraph of Subsection \ref{sec:conv}), by composition we obtain for free a bijection $\NC(\Sn, c)\overset{\sim}{\longrightarrow}\mathcal{I}(n)$ for any $c\in\Std(\Sn)$. More precisely, if $c\in\Std(\Sn)$ and $w\in\Sn$ are such that $w^{-1}cw=c_{\mathrm{lin}}$, then a bijection $\varphi_{w}$ is given by $$x\mapsto \varepsilon(w^{-1} x w).$$
Note that it depends on the choice of $w$ which is not necessarily unique, giving rise to several such bijections.
\end{rmq}

\subsection{An involution on the set of noncrossing partitions}\label{sub:involution}
The aim of this subsection is to define an involution $\inv:\NC(\Sn,c)\longrightarrow\NC(\Sn,c)$. To this end we associate to a noncrossing partition two additional sets of integers.

\begin{nota} Recall that $L_c$ is the set of elements of $[n+1]$ which label the left part of the circle. Set $M_x^{c}:=D_x^{c}\cap U_x^{c}\cap L_{c}$ and write $N_x^{c}$ for the set of indices which lie in $L_{c}$ but not in $\supp(x)$ and are nested in at least one polygon of $x$. In particular $M_x^{c}\cap N_x^{c}=\emptyset$ since elements of $M_x^c$ lie in $\supp(x)$ while elements of $N_x^c$ do not. 
\end{nota}

In other words, $M_x^c$ is the set of non-extremal indices of polygons of $x$ lying on the left part of the circle, while $N_x^c$ is the set of indices of the left part which are not vertices of polygons, but which are nested in at least one polygon of $x$. 

\begin{lemma}\label{lem:enlarge}
Let $x\in\NC(\Sn,c)$, $k\in N_x^c$. Among all the polygons in $\Pol(x)$ in which $k$ is nested, there is a unique one, say $P$, with the property that $P$ can be enlarged in a polygon $P'$ by adding to it the vertex $k$, and such that the resulting diagram with $P$ replaced by $P'$ stays noncrossing. 
\end{lemma}
\begin{proof}
The idea of the proof is illustrated in Figure \ref{fig:enlarge}. By definition of $N_x^c$, $k$ is on the left part of the circle and is nested in at least one polygon. Now our conventions on the graphical representation imply that $k$ is nested in $P\in\Pol(x)$ if and only if $P$ is such that the point $\min P$ (resp. the point $\max P$) is higher (resp. lower) than the point $k$. It implies that the polygons crossed by the horizontal line $\mathcal{L}_k$ of height $k$ (see Remark \ref{rmq:successifs}) are exactly those polygons in which $k$ is nested. The polygon $P$ we are looking for is necessarily the first one crossing $\mathcal{L}_k$ from the left. 
\end{proof}
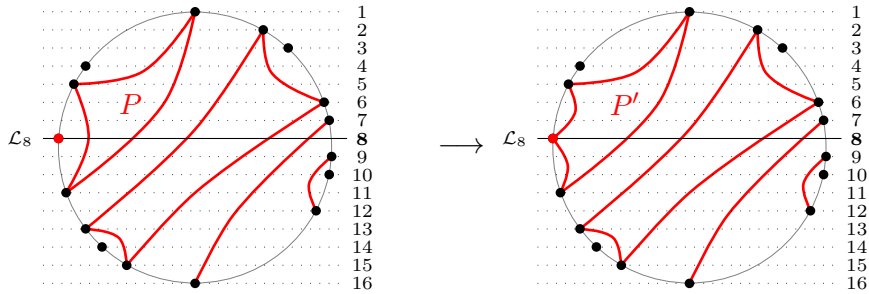
\begin{figure}[h!]
\begin{tabular}{ccc}
\begin{pspicture}(-2,-1.92)(2.5,1.92)
\pscircle[linecolor=gray, linewidth=0.2pt](0,0){1.8}

\pscurve[linewidth=1pt, linecolor=red](-1.59197,0.84)(-0.7,1)(0,1.8)
\pscurve[linewidth=1pt, linecolor=red](0,1.8)(-0.4,0.6)(-1.697,-0.6)
\pscurve[linewidth=1pt, linecolor=red](-1.59197,0.84)(-1.4,0.12)(-1.697,-0.6)

\pscurve[linewidth=1pt, linecolor=red](0.897,1.56)(1,1)(1.697,0.6)
\pscurve[linewidth=1pt, linecolor=red](1.697,0.6)(0, -0.6)(-0.897,-1.56)
\pscurve[linewidth=1pt, linecolor=red](-0.897,-1.56)(-1,-1.2)(-1.44,-1.08)
\pscurve[linewidth=1pt, linecolor=red](-1.44,-1.08)(0,0.24)(0.897,1.56)

\pscurve[linewidth=1pt, linecolor=red](1.7636,0.36)(0.5,-0.9)(0,-1.8)

\pscurve[linewidth=1pt, linecolor=red](1.796,-0.12)(1.5,-0.48)(1.59197,-0.84)




\psdots(0,1.8)(0.897,1.56)(1.2237,1.32)(-1.44,1.08)(-1.59197,0.84)(1.697,0.6)(1.7636,0.36)(-1.796,0.12)(1.796,-0.12)(1.7636,-0.36)(-1.697,-0.6)(1.59197,-0.84)(-1.44,-1.08)(-1.2237,-1.32)(-0.897,-1.56)(0,-1.8)

\psdot[linecolor=red, fillcolor=red](-1.796,0.12)
\rput(-2.3, 0.12){\tiny $\mathcal{L}_8$}

\psline[linestyle=dotted, linewidth=0.4pt](-2,1.8)(2,1.8)
\psline[linestyle=dotted, linewidth=0.4pt](-2,1.56)(2,1.56)
\psline[linestyle=dotted, linewidth=0.4pt](-2,1.32)(2,1.32)
\psline[linestyle=dotted, linewidth=0.4pt](-2,1.08)(2,1.08)
\psline[linestyle=dotted, linewidth=0.4pt](-2,0.84)(2,0.84)
\psline[linestyle=dotted, linewidth=0.4pt](-2,0.6)(2,0.6)
\psline[linestyle=dotted, linewidth=0.4pt](-2,0.36)(2,0.36)
\psline[linewidth=0.4pt](-2,0.12)(2,0.12)
\psline[linestyle=dotted, linewidth=0.4pt](-2,-1.8)(2,-1.8)
\psline[linestyle=dotted, linewidth=0.4pt](-2,-1.56)(2,-1.56)
\psline[linestyle=dotted, linewidth=0.4pt](-2,-1.32)(2,-1.32)
\psline[linestyle=dotted, linewidth=0.4pt](-2,-1.08)(2,-1.08)
\psline[linestyle=dotted, linewidth=0.4pt](-2,-0.84)(2,-0.84)
\psline[linestyle=dotted, linewidth=0.4pt](-2,-0.6)(2,-0.6)
\psline[linestyle=dotted, linewidth=0.4pt](-2,-0.36)(2,-0.36)
\psline[linestyle=dotted, linewidth=0.4pt](-2,-0.12)(2,-0.12)

\psdot[linecolor=red, fillcolor=red](-1.796,0.12)

\rput(-0.85,0.58){\red $P$}

\rput(2.2, 1.8){\tiny $1$}
\rput(2.2, 1.56){\tiny $2$}
\rput(2.2, 1.32){\tiny $3$}
\rput(2.2, 1.08){\tiny $4$}
\rput(2.2, 0.84){\tiny $5$}
\rput(2.2, 0.6){\tiny $6$}
\rput(2.2, 0.36){\tiny $7$}
\rput(2.2, 0.12){\tiny $\mathbf{8}$}
\rput(2.2, -0.12){\tiny $9$}
\rput(2.2, -0.36){\tiny $10$}
\rput(2.2, -0.6){\tiny $11$}
\rput(2.2, -0.84){\tiny $12$}
\rput(2.2, -1.08){\tiny $13$}
\rput(2.2, -1.32){\tiny $14$}
\rput(2.2, -1.56){\tiny $15$}
\rput(2.2, -1.8){\tiny $16$}
\end{pspicture}
& \begin{pspicture}(-0.65,-1.92)(0.65,1.92)
\rput(0,0){$~~\longrightarrow~~$}
\end{pspicture}
 &
\begin{pspicture}(-2,-1.92)(2.5,1.92)
\pscircle[linecolor=gray, linewidth=0.2pt](0,0){1.8}

\pscurve[linewidth=1pt, linecolor=red](-1.59197,0.84)(-0.7,1)(0,1.8)
\pscurve[linewidth=1pt, linecolor=red](0,1.8)(-0.4,0.6)(-1.697,-0.6)
\pscurve[linewidth=1pt, linecolor=red](-1.796,0.12)(-1.6,-0.24)(-1.697,-0.6)
\pscurve[linewidth=1pt, linecolor=red](-1.796,0.12)(-1.5,0.48)(-1.59197,0.84)

\pscurve[linewidth=1pt, linecolor=red](0.897,1.56)(1,1)(1.697,0.6)
\pscurve[linewidth=1pt, linecolor=red](1.697,0.6)(0, -0.6)(-0.897,-1.56)
\pscurve[linewidth=1pt, linecolor=red](-0.897,-1.56)(-1,-1.2)(-1.44,-1.08)
\pscurve[linewidth=1pt, linecolor=red](-1.44,-1.08)(0,0.24)(0.897,1.56)

\pscurve[linewidth=1pt, linecolor=red](1.7636,0.36)(0.5,-0.9)(0,-1.8)

\pscurve[linewidth=1pt, linecolor=red](1.796,-0.12)(1.5,-0.48)(1.59197,-0.84)




\psdots(0,1.8)(0.897,1.56)(1.2237,1.32)(-1.44,1.08)(-1.59197,0.84)(1.697,0.6)(1.7636,0.36)(-1.796,0.12)(1.796,-0.12)(1.7636,-0.36)(-1.697,-0.6)(1.59197,-0.84)(-1.44,-1.08)(-1.2237,-1.32)(-0.897,-1.56)(0,-1.8)

\psdot[linecolor=red, fillcolor=red](-1.796,0.12)
\rput(-2.3, 0.12){\tiny $\mathcal{L}_8$}

\psline[linestyle=dotted, linewidth=0.4pt](-2,1.8)(2,1.8)
\psline[linestyle=dotted, linewidth=0.4pt](-2,1.56)(2,1.56)
\psline[linestyle=dotted, linewidth=0.4pt](-2,1.32)(2,1.32)
\psline[linestyle=dotted, linewidth=0.4pt](-2,1.08)(2,1.08)
\psline[linestyle=dotted, linewidth=0.4pt](-2,0.84)(2,0.84)
\psline[linestyle=dotted, linewidth=0.4pt](-2,0.6)(2,0.6)
\psline[linestyle=dotted, linewidth=0.4pt](-2,0.36)(2,0.36)
\psline[linewidth=0.4pt](-2,0.12)(2,0.12)
\psline[linestyle=dotted, linewidth=0.4pt](-2,-1.8)(2,-1.8)
\psline[linestyle=dotted, linewidth=0.4pt](-2,-1.56)(2,-1.56)
\psline[linestyle=dotted, linewidth=0.4pt](-2,-1.32)(2,-1.32)
\psline[linestyle=dotted, linewidth=0.4pt](-2,-1.08)(2,-1.08)
\psline[linestyle=dotted, linewidth=0.4pt](-2,-0.84)(2,-0.84)
\psline[linestyle=dotted, linewidth=0.4pt](-2,-0.6)(2,-0.6)
\psline[linestyle=dotted, linewidth=0.4pt](-2,-0.36)(2,-0.36)
\psline[linestyle=dotted, linewidth=0.4pt](-2,-0.12)(2,-0.12)

\psdot[linecolor=red, fillcolor=red](-1.796,0.12)

\rput(-0.85,0.58){\red $P'$}

\rput(2.2, 1.8){\tiny $1$}
\rput(2.2, 1.56){\tiny $2$}
\rput(2.2, 1.32){\tiny $3$}
\rput(2.2, 1.08){\tiny $4$}
\rput(2.2, 0.84){\tiny $5$}
\rput(2.2, 0.6){\tiny $6$}
\rput(2.2, 0.36){\tiny $7$}
\rput(2.2, 0.12){\tiny $\mathbf{8}$}
\rput(2.2, -0.12){\tiny $9$}
\rput(2.2, -0.36){\tiny $10$}
\rput(2.2, -0.6){\tiny $11$}
\rput(2.2, -0.84){\tiny $12$}
\rput(2.2, -1.08){\tiny $13$}
\rput(2.2, -1.32){\tiny $14$}
\rput(2.2, -1.56){\tiny $15$}
\rput(2.2, -1.8){\tiny $16$}
\end{pspicture}
\end{tabular}
\caption{Picture for the proof of Lemma \ref{lem:enlarge}, with $k=8$.}
\label{fig:enlarge}
\end{figure}

The idea is then to enlarge the polygons of $x$ by adding all the vertices with labels in $N_x^c$ using Lemma \ref{lem:enlarge} inductively; such a process is well-defined since in the notations of Lemma \ref{lem:enlarge}, an index $i\neq k$ is nested in $P$ if and only if it is nested in $P'$; moreover, $P'$ is the leftmost polygon in which $i$ is nested in the noncrossing partition $x'$ (obtained by adding $k$ to $x$) if and only if $P$ is the leftmost polygon in which $i$ is nested in the graphical representation of $x$. In particular it follows easily that the order in which we add the various vertices with labels in $N_x^c$ does not affect the result. The diagram obtained after adding all the vertices with labels in $N_x^c$ is noncrossing, hence it represents a unique $x'\in\NC(\Sn,c)$. The first part of Figure \ref{fig:involution} gives an example of the map $x\mapsto x'$.

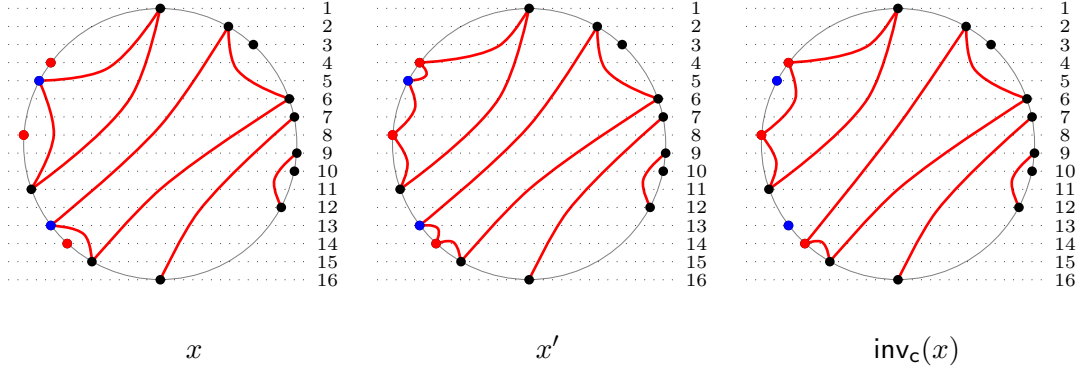
\begin{figure}[h!]
\begin{tabular}{ccc}
\begin{pspicture}(-1.6,-1.92)(2.5,1.92)
\pscircle[linecolor=gray, linewidth=0.2pt](0,0){1.8}

\pscurve[linewidth=1pt, linecolor=red](-1.59197,0.84)(-0.7,1)(0,1.8)
\pscurve[linewidth=1pt, linecolor=red](0,1.8)(-0.4,0.6)(-1.697,-0.6)
\pscurve[linewidth=1pt, linecolor=red](-1.59197,0.84)(-1.4,0.12)(-1.697,-0.6)

\pscurve[linewidth=1pt, linecolor=red](0.897,1.56)(1,1)(1.697,0.6)
\pscurve[linewidth=1pt, linecolor=red](1.697,0.6)(0, -0.6)(-0.897,-1.56)
\pscurve[linewidth=1pt, linecolor=red](-0.897,-1.56)(-1,-1.2)(-1.44,-1.08)
\pscurve[linewidth=1pt, linecolor=red](-1.44,-1.08)(0,0.24)(0.897,1.56)

\pscurve[linewidth=1pt, linecolor=red](1.7636,0.36)(0.5,-0.9)(0,-1.8)

\pscurve[linewidth=1pt, linecolor=red](1.796,-0.12)(1.5,-0.48)(1.59197,-0.84)




\psdots(0,1.8)(0.897,1.56)(1.2237,1.32)(-1.44,1.08)(-1.59197,0.84)(1.697,0.6)(1.7636,0.36)(-1.796,0.12)(1.796,-0.12)(1.7636,-0.36)(-1.697,-0.6)(1.59197,-0.84)(-1.44,-1.08)(-1.2237,-1.32)(-0.897,-1.56)(0,-1.8)

\psdot[linecolor=red, fillcolor=red](-1.796,0.12)

\psline[linestyle=dotted, linewidth=0.4pt](-2,1.8)(2,1.8)
\psline[linestyle=dotted, linewidth=0.4pt](-2,1.56)(2,1.56)
\psline[linestyle=dotted, linewidth=0.4pt](-2,1.32)(2,1.32)
\psline[linestyle=dotted, linewidth=0.4pt](-2,1.08)(2,1.08)
\psline[linestyle=dotted, linewidth=0.4pt](-2,0.84)(2,0.84)
\psline[linestyle=dotted, linewidth=0.4pt](-2,0.6)(2,0.6)
\psline[linestyle=dotted, linewidth=0.4pt](-2,0.36)(2,0.36)
\psline[linestyle=dotted, linewidth=0.4pt](-2,0.12)(2,0.12)
\psline[linestyle=dotted, linewidth=0.4pt](-2,-1.8)(2,-1.8)
\psline[linestyle=dotted, linewidth=0.4pt](-2,-1.56)(2,-1.56)
\psline[linestyle=dotted, linewidth=0.4pt](-2,-1.32)(2,-1.32)
\psline[linestyle=dotted, linewidth=0.4pt](-2,-1.08)(2,-1.08)
\psline[linestyle=dotted, linewidth=0.4pt](-2,-0.84)(2,-0.84)
\psline[linestyle=dotted, linewidth=0.4pt](-2,-0.6)(2,-0.6)
\psline[linestyle=dotted, linewidth=0.4pt](-2,-0.36)(2,-0.36)
\psline[linestyle=dotted, linewidth=0.4pt](-2,-0.12)(2,-0.12)

\psdot[linecolor=red, fillcolor=red](-1.796,0.12)
\psdot[linecolor=red, fillcolor=red](-1.44,1.08)
\psdot[linecolor=red, fillcolor=red](-1.2237,-1.32)

\psdot[linecolor=blue, fillcolor=blue](-1.59197,0.84)
\psdot[linecolor=blue, fillcolor=blue](-1.44,-1.08)

\rput(2.2, 1.8){\tiny $1$}
\rput(2.2, 1.56){\tiny $2$}
\rput(2.2, 1.32){\tiny $3$}
\rput(2.2, 1.08){\tiny $4$}
\rput(2.2, 0.84){\tiny $5$}
\rput(2.2, 0.6){\tiny $6$}
\rput(2.2, 0.36){\tiny $7$}
\rput(2.2, 0.12){\tiny $8$}
\rput(2.2, -0.12){\tiny $9$}
\rput(2.2, -0.36){\tiny $10$}
\rput(2.2, -0.6){\tiny $11$}
\rput(2.2, -0.84){\tiny $12$}
\rput(2.2, -1.08){\tiny $13$}
\rput(2.2, -1.32){\tiny $14$}
\rput(2.2, -1.56){\tiny $15$}
\rput(2.2, -1.8){\tiny $16$}
\end{pspicture}
& 
\begin{pspicture}(-2,-1.92)(2.5,1.92)
\pscircle[linecolor=gray, linewidth=0.2pt](0,0){1.8}

\pscurve[linewidth=1pt, linecolor=red](-1.59197,0.84)(-1.35,0.92)(-1.44,1.08)
\pscurve[linewidth=1pt, linecolor=red](-1.44,1.08)(-0.4,1.3)(0,1.8)
\pscurve[linewidth=1pt, linecolor=red](0,1.8)(-0.4,0.6)(-1.697,-0.6)
\pscurve[linewidth=1pt, linecolor=red](-1.796,0.12)(-1.6,-0.24)(-1.697,-0.6)
\pscurve[linewidth=1pt, linecolor=red](-1.796,0.12)(-1.5,0.48)(-1.59197,0.84)

\pscurve[linewidth=1pt, linecolor=red](0.897,1.56)(1,1)(1.697,0.6)
\pscurve[linewidth=1pt, linecolor=red](1.697,0.6)(0, -0.6)(-0.897,-1.56)
\pscurve[linewidth=1pt, linecolor=red](-0.897,-1.56)(-1, -1.3)(-1.2237,-1.32)
\pscurve[linewidth=1pt, linecolor=red](-1.2237,-1.32)(-1.2,-1.12)(-1.44,-1.08)
\pscurve[linewidth=1pt, linecolor=red](-1.44,-1.08)(0,0.24)(0.897,1.56)

\pscurve[linewidth=1pt, linecolor=red](1.7636,0.36)(0.5,-0.9)(0,-1.8)

\pscurve[linewidth=1pt, linecolor=red](1.796,-0.12)(1.5,-0.48)(1.59197,-0.84)




\psdots(0,1.8)(0.897,1.56)(1.2237,1.32)(-1.44,1.08)(-1.59197,0.84)(1.697,0.6)(1.7636,0.36)(-1.796,0.12)(1.796,-0.12)(1.7636,-0.36)(-1.697,-0.6)(1.59197,-0.84)(-1.44,-1.08)(-1.2237,-1.32)(-0.897,-1.56)(0,-1.8)

\psdot[linecolor=red, fillcolor=red](-1.796,0.12)

\psline[linestyle=dotted, linewidth=0.4pt](-2,1.8)(2,1.8)
\psline[linestyle=dotted, linewidth=0.4pt](-2,1.56)(2,1.56)
\psline[linestyle=dotted, linewidth=0.4pt](-2,1.32)(2,1.32)
\psline[linestyle=dotted, linewidth=0.4pt](-2,1.08)(2,1.08)
\psline[linestyle=dotted, linewidth=0.4pt](-2,0.84)(2,0.84)
\psline[linestyle=dotted, linewidth=0.4pt](-2,0.6)(2,0.6)
\psline[linestyle=dotted, linewidth=0.4pt](-2,0.36)(2,0.36)
\psline[linestyle=dotted, linewidth=0.4pt](-2,0.12)(2,0.12)
\psline[linestyle=dotted, linewidth=0.4pt](-2,-1.8)(2,-1.8)
\psline[linestyle=dotted, linewidth=0.4pt](-2,-1.56)(2,-1.56)
\psline[linestyle=dotted, linewidth=0.4pt](-2,-1.32)(2,-1.32)
\psline[linestyle=dotted, linewidth=0.4pt](-2,-1.08)(2,-1.08)
\psline[linestyle=dotted, linewidth=0.4pt](-2,-0.84)(2,-0.84)
\psline[linestyle=dotted, linewidth=0.4pt](-2,-0.6)(2,-0.6)
\psline[linestyle=dotted, linewidth=0.4pt](-2,-0.36)(2,-0.36)
\psline[linestyle=dotted, linewidth=0.4pt](-2,-0.12)(2,-0.12)

\psdot[linecolor=red, fillcolor=red](-1.796,0.12)
\psdot[linecolor=red, fillcolor=red](-1.44,1.08)
\psdot[linecolor=red, fillcolor=red](-1.2237,-1.32)

\psdot[linecolor=blue, fillcolor=blue](-1.59197,0.84)
\psdot[linecolor=blue, fillcolor=blue](-1.44,-1.08)

\rput(2.2, 1.8){\tiny $1$}
\rput(2.2, 1.56){\tiny $2$}
\rput(2.2, 1.32){\tiny $3$}
\rput(2.2, 1.08){\tiny $4$}
\rput(2.2, 0.84){\tiny $5$}
\rput(2.2, 0.6){\tiny $6$}
\rput(2.2, 0.36){\tiny $7$}
\rput(2.2, 0.12){\tiny $8$}
\rput(2.2, -0.12){\tiny $9$}
\rput(2.2, -0.36){\tiny $10$}
\rput(2.2, -0.6){\tiny $11$}
\rput(2.2, -0.84){\tiny $12$}
\rput(2.2, -1.08){\tiny $13$}
\rput(2.2, -1.32){\tiny $14$}
\rput(2.2, -1.56){\tiny $15$}
\rput(2.2, -1.8){\tiny $16$}
\end{pspicture}&
\begin{pspicture}(-2,-1.92)(2.5,1.92)
\pscircle[linecolor=gray, linewidth=0.2pt](0,0){1.8}

\pscurve[linewidth=1pt, linecolor=red](-1.44,1.08)(-0.4,1.3)(0,1.8)
\pscurve[linewidth=1pt, linecolor=red](0,1.8)(-0.4,0.6)(-1.697,-0.6)
\pscurve[linewidth=1pt, linecolor=red](-1.796,0.12)(-1.6,-0.24)(-1.697,-0.6)
\pscurve[linewidth=1pt, linecolor=red](-1.796,0.12)(-1.35,0.6)(-1.44,1.08)

\pscurve[linewidth=1pt, linecolor=red](0.897,1.56)(1,1)(1.697,0.6)
\pscurve[linewidth=1pt, linecolor=red](1.697,0.6)(0, -0.6)(-0.897,-1.56)
\pscurve[linewidth=1pt, linecolor=red](-0.897,-1.56)(-1, -1.3)(-1.2237,-1.32)
\pscurve[linewidth=1pt, linecolor=red](-1.2237,-1.32)(0,0.24)(0.897,1.56)

\pscurve[linewidth=1pt, linecolor=red](1.7636,0.36)(0.5,-0.9)(0,-1.8)

\pscurve[linewidth=1pt, linecolor=red](1.796,-0.12)(1.5,-0.48)(1.59197,-0.84)




\psdots(0,1.8)(0.897,1.56)(1.2237,1.32)(-1.44,1.08)(-1.59197,0.84)(1.697,0.6)(1.7636,0.36)(-1.796,0.12)(1.796,-0.12)(1.7636,-0.36)(-1.697,-0.6)(1.59197,-0.84)(-1.44,-1.08)(-1.2237,-1.32)(-0.897,-1.56)(0,-1.8)

\psdot[linecolor=red, fillcolor=red](-1.796,0.12)

\psline[linestyle=dotted, linewidth=0.4pt](-2,1.8)(2,1.8)
\psline[linestyle=dotted, linewidth=0.4pt](-2,1.56)(2,1.56)
\psline[linestyle=dotted, linewidth=0.4pt](-2,1.32)(2,1.32)
\psline[linestyle=dotted, linewidth=0.4pt](-2,1.08)(2,1.08)
\psline[linestyle=dotted, linewidth=0.4pt](-2,0.84)(2,0.84)
\psline[linestyle=dotted, linewidth=0.4pt](-2,0.6)(2,0.6)
\psline[linestyle=dotted, linewidth=0.4pt](-2,0.36)(2,0.36)
\psline[linestyle=dotted, linewidth=0.4pt](-2,0.12)(2,0.12)
\psline[linestyle=dotted, linewidth=0.4pt](-2,-1.8)(2,-1.8)
\psline[linestyle=dotted, linewidth=0.4pt](-2,-1.56)(2,-1.56)
\psline[linestyle=dotted, linewidth=0.4pt](-2,-1.32)(2,-1.32)
\psline[linestyle=dotted, linewidth=0.4pt](-2,-1.08)(2,-1.08)
\psline[linestyle=dotted, linewidth=0.4pt](-2,-0.84)(2,-0.84)
\psline[linestyle=dotted, linewidth=0.4pt](-2,-0.6)(2,-0.6)
\psline[linestyle=dotted, linewidth=0.4pt](-2,-0.36)(2,-0.36)
\psline[linestyle=dotted, linewidth=0.4pt](-2,-0.12)(2,-0.12)

\psdot[linecolor=red, fillcolor=red](-1.796,0.12)
\psdot[linecolor=red, fillcolor=red](-1.44,1.08)
\psdot[linecolor=red, fillcolor=red](-1.2237,-1.32)

\psdot[linecolor=blue, fillcolor=blue](-1.59197,0.84)
\psdot[linecolor=blue, fillcolor=blue](-1.44,-1.08)

\rput(2.2, 1.8){\tiny $1$}
\rput(2.2, 1.56){\tiny $2$}
\rput(2.2, 1.32){\tiny $3$}
\rput(2.2, 1.08){\tiny $4$}
\rput(2.2, 0.84){\tiny $5$}
\rput(2.2, 0.6){\tiny $6$}
\rput(2.2, 0.36){\tiny $7$}
\rput(2.2, 0.12){\tiny $8$}
\rput(2.2, -0.12){\tiny $9$}
\rput(2.2, -0.36){\tiny $10$}
\rput(2.2, -0.6){\tiny $11$}
\rput(2.2, -0.84){\tiny $12$}
\rput(2.2, -1.08){\tiny $13$}
\rput(2.2, -1.32){\tiny $14$}
\rput(2.2, -1.56){\tiny $15$}
\rput(2.2, -1.8){\tiny $16$}
\end{pspicture}\\
& &\\
 $x$ & $x'$ & $\inv(x)$\\
\end{tabular}
\caption{The involution $x\mapsto\inv(x)$. We have $M_x^c=\{5, 13\}$, $N_x^c=\{4,8,14\}$.}
\label{fig:involution}
\end{figure}

We then consider the noncrossing partition $\inv(x)$ obtained from $x'$ by removing from any polygon $P\in\Pol(x')$ the vertices with index in $M_x^{c}$, that is, the non extremal indices of polygons of $x$ which moreover lie in $L_{c}$. These indices become nested in polygons of $\inv(x)$ and they do not lie in $\Vert(\inv({x}))$. See Figure~\ref{fig:involution} for an example. In fact, by construction we have $$N_{\inv(x)}^{c}=M_x^{c}\text{ and }M_{\inv(x)}^{c}=N_x^{c}.$$ We therefore get

\begin{lemma}\label{lem:invol}
The map $\inv:\NC(\Sn,c)\longrightarrow\NC(\Sn,c)$ is an involution such that for all $x\in\NC(\Sn,c)$ we have $$(D_{\inv(x)}^{c}, U_{\inv(x)}^{c})=((D_{x}^{c}\backslash M_x^{c})\cup N_x^{c}, (U_{x}^{c}\backslash M_x^{c})\cup N_x^{c}).$$
Moreover $\inv$ preserves both sets of initial and terminal indices and we have $|\Pol(x)|=|\Pol(\inv(x))|$.
\end{lemma}

\begin{rmq}\label{rmq:triv}
For $c=c_{\mathrm{lin}}$ we have $\inv=\mathrm{id}$; indeed $L_c=\{1,n+1\}$ in that case, but the points $1$ and $n+1$ cannot be nested in a polygon and if they index a polygon, they are either initial or terminal. It follows that for all $x\in\NC(\Sn,c)$ we have $M_x^c=\emptyset=N_x^c$. In case $c=c_{\mathrm{lin}}^{-1}$ we have $\inv(c)=(1,n+1)$. It would be interesting to find a Coxeter-theoretic interpretation of $\inv$. It is easy to see that the map $\inv$ preserves the $\S$-support, that is, we have ${{\sf{supp}}}_{\mathcal{S}}(x)={{\sf{supp}}}_{\mathcal{S}}(\inv(x))$ for all $x\in\NC(\Sn,c)$. Note that we can similarly define an involution by replacing the set $L_c$ by $R_c$ in the definitions, in which case we obtain the identity for $c=c_{\mathrm{lin}}^{-1}$ and an involution sending $c$ to the transposition $(1,n+1)$ for $c=c_{\mathrm{lin}}$.
\end{rmq}


\begin{rmq}
Given $x\in\NC(\Sn, c)$, any cycle $c_i$ from the decomposition of $x=c_1 c_2\cdots c_k$ is in $\NC(\Sn,c)$, hence $\inv(c_i)\in\NC(\Sn,c)$. However in general $\inv(x)\neq\inv(c_1) \inv(c_2)\cdots \inv(c_k)$. It can even happen that the element in the right hand side does not lie in $\NC(\Sn,c)$. For example, for $c=(1,4,5,3,2)$, $w=(1,4)(2,5)\in\NC(\mathfrak{S}_5, c)$ we have $\inv((1,4))\inv((2,5))=(1,4,3,2)(2,5,3)=(1,4,3)(2,5)\notin\NC(\mathfrak{S}_5,c)$. \end{rmq}

\subsection{Bijection $\#$2}\label{sub:bij2}

\begin{nota}
We define a bijection $\psi_2:\NC(\Sn,c)\longrightarrow \mathcal{I}(n)$ by twisting the bijection $\psi_1$ from Proposition~\ref{prop:gobwilbij} by the involution $\inv$, i.e., for $x\in\NC(\Sn,c)$ we set $\psi_2(x):=\psi_1(\inv(x))$.
\end{nota}

\subsection{Comparison of the bijections}\label{ref:comparison}
In this subsection, we compare the bijections $\psi_1, \psi_2:\NC(\Sn,c)\overset{\sim}{\longrightarrow} \mathcal{I}(n)$ given in Subsections \ref{sub:bij1} and \ref{sub:bij2}. We compute them in Figure~\ref{fig:difftes} for $n=4$, $c=s_2 s_1 s_3=(1,3,4,2)$. We have $L_c=\{4,2,1\}$, $R_c=\{1,3,4\}$. There are $14$ elements in $\NC(\Sn,c)$ and $\NC(\Sn,c_{\mathrm{lin}})$ and they are depicted by their cycle decomposition.  

One sees that if $\psi_1(x)=(D,U)$ and $\psi_2(x)=(D',U')$, then $$D\backslash (D\cap U)=D'\backslash (D'\cap U')\text{ and }U\backslash (D\cap U)=U'\backslash (D'\cap U').$$ This follows from the fact that $\inv$ preserves the initial and terminal indices of $x$ which are exactly those lying in these sets. 

\begin{figure}
\begin{center}
{
\footnotesize
\begin{pspicture}(-2.6,-3)(6,4)

\begin{tabular}{|c|c|c|c|}
\thickhline
$x$ & $\inv(x)$ & $\psi_1(x)$ & $\psi_2(x)$\\
\thickhline
$\mathrm{id}$ & $\mathrm{id}$ & $(\emptyset,\emptyset)$ & $(\emptyset,\emptyset)$\\
\hline
$(1,2)$ & $(1,2)$ & $(\{1\},\{2\})$ & $(\{1\},\{2\})$ \\
\hline
$(2,3)$ & $(2,3)$ & $(\{2\},\{3\})$ & $(\{2\},\{3\})$\\
\hline
$(3,4)$ & $(3,4)$ & $(\{3\},\{4\})$ & $(\{3\},\{4\})$ \\
\hline
$(1,3)$ & $(2,1,3)$ & $(\{1\},\{3\})$ & $(\{1,2\},\{2,3\})$ \\
\hline
$(2,4)$ & $(2,4)$ & $(\{2\},\{4\})$ & $(\{2\},\{4\})$ \\
\hline
$(1,4)$ & $(1,4,2)$ & $(\{1\},\{4\})$ & $(\{1,2\},\{2,4\})$ \\
\hline
$(1,3,4)$ & $(1,3,4,2)$ & $(\{1,3\},\{3,4\})$ & $(\{1,2,3\},\{2,3,4\})$ \\
\hline
$(3,4,2)$ & $(3,4,2)$ & $(\{2,3\},\{3,4\})$ & $(\{2,3\},\{3,4\})$ \\
\hline
$(4,2,1)$ & $(1,4)$ & $(\{1,2\},\{2,4\})$ & $(\{1\},\{4\})$ \\
\hline
$(2,1,3)$ & $(1,3)$ & $(\{1,2\},\{2,3\})$ & $(\{1\},\{3\})$ \\
\hline
$(1,3)(2,4)$ & $(1,3)(2,4)$ & $(\{1,2\},\{3,4\})$ & $(\{1,2\},\{3,4\})$ \\
\hline
$(1,2)(3,4)$ & $(1,2)(3,4)$ & $(\{1,3\},\{2,4\})$ & $(\{1,3\},\{2,4\})$ \\
\hline
$(1,3,4,2)$ & $(1,3,4)$ & $(\{1,2,3\},\{2,3,4\})$ & $(\{1,3\},\{3,4\})$\\
\thickhline

\end{tabular}

\end{pspicture}
}
\end{center}
\caption{The bijections $\psi_1, \psi_2:\NC(\Sn,c)\overset{\sim}{\longrightarrow} \mathcal{I}(n)$ for $n=4$ and $c=s_2 s_1 s_3=(1,3,4,2)$.}
\label{fig:difftes}
\end{figure}

\subsection{Bijections with fully commutative elements}\label{sub:bijwithful}

In this subsection, we introduce the set $\F(\Sn)$ of fully commutative elements of $\Sn$ and explain how to use $\psi_2$ (or any bijection $\NC(\Sn,c)\overset{\sim}{\longrightarrow} \mathcal{I}(n)$) to build a bijection $\varphi:\NC(\Sn,c)\longrightarrow \F(\Sn)$ for any $c\in\Std(\Sn)$.

\begin{definition}\label{defn:fully}

An element $w\in \Sn$ is \defn{fully commutative} if any two $\S$-reduced expressions of $w$ can be related by a sequence of commutation of adjacent letters (these are also known as the \defn{$321$-avoiding permutations}, see~\cite{stem} or item $115$ in \cite{Stanley}; in particular they are known to have enumeration $C_{n+1}$). Note that this definition can be given for an abritrary Coxeter group. We denote by $\F(\Sn)$ the set of fully commutative elements of $\Sn$. Each $w\in\F(\Sn)$ has a unique $\S$-reduced expression of the form \[(s_{i_1} s_{i_1-1}\cdots s_{j_1})(s_{i_2} s_{i_2-1}\cdots s_{j_2})\cdots (s_{i_k} s_{i_k-1}\cdots s_{j_k}),\]
where all the indices lie in $\{1,\dots,n\}$, $i_1<i_2<\dots<i_k$, $j_1<j_2<\dots<j_k$ and $j_m\leq i_m$ for all $1\leq m\leq k$. Conversely, any element written in this form is fully commutative. The $\S$-reduced expression above is called the \defn{normal form} of $w$. We set $I_w:=\{i_1,\dots, i_k\}$, $J_w:=\{j_1,\dots, j_k\}$.  

\end{definition}
The following is an easy consequence of the existence of normal forms (see for instance~\cite[Corollary 2.3]{Gob}) which will be needed further:

\begin{lemma}\label{lem:ij}
Let $w\in\F(\Sn)$. Let $i\in\{1,\dots, n\}$ such that $s_i$ occurs in a (equivalently any) $\S$-reduced expression of $w$. Then 
\begin{enumerate}
\item $i\in I_w$ if and only if in any $\S$-reduced expression of $w$, there is no occurrence of $s_{i+1}$ before the first occurrence of $s_i$,
\item $i\in J_w$ if and only if in any $\S$-reduced expression of $w$, there is no occurrence of $s_{i-1}$ after the last occurrence of $s_i$. 
\end{enumerate}
\end{lemma}
Also note the following characterization of fully commutative elements in $\Sn$ (see for instance \cite[Proposition 2.2]{Gob}):
\begin{proposition}\label{prop:fullya}
Let $w\in\Sn$. The following are equivalent:
\begin{enumerate}
\item The element $w$ is fully commutative,
\item If $s_{i_1}\cdots s_{i_k}$ is an $\S$-reduced expression of $w$, then for all $i=1,\dots,n$, the integer $n_i(w):=|\{j~|~i_j=i\}|$ is independent of the chosen $\S$-reduced expression. 
\end{enumerate}
\end{proposition}

\begin{lemma}[{\cite[Theorem 1]{Z}}]\label{lem:fullocc}
An $\S$-word $s_{i_1}\cdots s_{i_k}$ is a reduced expression of a fully commutative element if and only if for any $i$ such that $s_i$ occurs at least twice in the word, there is exactly one occurrence of $s_{i-1}$ and one occurrence of $s_{i+1}$ between any two successive occurrences of $s_i$. 
\end{lemma}

We denote by $I_w[1]$ the set $I_w+1:=\{i_1+1,\dots, i_k+1\}$. It follows directly from the existence and unicity of normal forms that the map $$g:\mathcal{I}(n)\longrightarrow \F(\Sn),~(J_w, I_w[1])\mapsto w$$ is a bijection. Composing with the bijection $\psi_2$ from Subsection~\ref{sub:involution} we get 

\begin{theorem}\label{thm:bijecgenerales}
Let $c\in\Std(\Sn)$. The map $$\varphi:=g\circ \psi_2:\NC(\Sn,c)\longrightarrow \F(\Sn)$$ is a bijection. It is characterized in terms of the combinatorial data as follows: let $x\in\NC(\Sn,x)$, $k\in[n+1]$. Then 
$$k\in J_{\varphi(x)}\Leftrightarrow\left\{
    \begin{array}{l}
        k\in R_{c}\cap D_x^{c}\mbox{, or}\\
        k\in L_{c}\mbox{ and }k\mbox{ is an initial index of a polygon of }x\mbox{, or}\\
        k\in L_{c}\mbox{ and }k\notin\supp(x)\mbox{ and }k\mbox{ is nested in a polygon of }x.\\ 
    \end{array}
    \right.
$$
We also have that
$$k-1\in I_{\varphi(x)}\Leftrightarrow
\left\{
    \begin{array}{l}
        k\in R_{c}\cap U_x^{c}\mbox{, or}\\
        k\in L_{c}\mbox{ and }k\mbox{ is a terminal index of a polygon of }x\mbox{, or}\\
        k\in L_{c}\mbox{ and }k\notin\supp(x)\mbox{ and $k$ is nested in a polygon of }x.\\ 
    \end{array}
    \right.
$$
\end{theorem}

\begin{proof}
The map is a bijection since both $g$ and $\psi_2$ are. We have by definition of $\psi_2$ that $J_{\varphi(x)}=D_{\inv({x})}^c$, $I_{\varphi(x)}[1]=U_{\inv({x})}^c$. The rest of the claim follows from the definition and properties of the involution $\inv$.
\end{proof}
In general we use the letter $x$ for a noncrossing partition and the letter $w$ for a fully commutative element.

\begin{exple}\label{ex:bijarbitraire}
We continue the example given in Subsection~\ref{ref:comparison} and explicitly give the map $\varphi$ for $c=(1,3,4,2)$ in Figure~\ref{fig:thebijection}. Recall that for $x\in\NC(\Sn,c)$ we have $$\psi_2(x)=(D_{\inv({x})}^c, U_{\inv({x})}^c)=(J_{\varphi(x)}, I_{\varphi(x)}[1]).$$

\begin{center}
\begin{figure}
{
\footnotesize
\begin{tabular}{|c|c|c|c|c|}
\thickhline
$x$ & $\psi_2(x)$ & $(J_{\varphi(x)}, I_{\varphi(x)})$ & $\varphi(x)$ (n.f.) & $\varphi(x)$ (cycle dec.) \\
\thickhline
$\mathrm{id}$ & $(\emptyset,\emptyset)$ & $(\emptyset,\emptyset)$ & $e$ & $\mathrm{id}$\\
\hline
$(1,2)$ &  $(\{1\},\{2\})$ & $(\{1\},\{1\})$ & $s_1$ & $(1,2)$ \\
\hline
$(2,3)$ &  $(\{2\},\{3\})$ & $(\{2\},\{2\})$ & $s_2$ & $(2,3)$ \\
\hline
$(3,4)$  &  $(\{3\},\{4\})$ & $(\{3\},\{3\})$ & $s_3$ & $(3,4)$\\
\hline
$(1,3)$  &  $(\{1,2\},\{2,3\})$ & $(\{1,2\},\{1,2\})$ & $(s_1)(s_2)$ & $(1,2,3)$\\
\hline
$(2,4)$  & $(\{2\},\{4\})$ & $(\{2\},\{3\})$ & $s_3s_2$ & $(4,3,2)$\\
\hline
$(1,4)$  & $(\{1,2\},\{2,4\})$ & $(\{1,2\},\{1,3\})$ & $(s_1) (s_3 s_2)$ & $(1,2,4,3)$ \\
\hline
$(1,3,4)$  & $(\{1,2,3\},\{2,3,4\})$ & $(\{1,2,3\},\{1,2,3\})$ & $(s_1)( s_2)( s_3)$ & $(1,2,3,4)$\\
\hline
$(3,4,2)$  & $(\{2,3\},\{3,4\})$ & $(\{2,3\},\{2,3\})$ & $(s_2) (s_3)$ & $(2,3,4)$\\
\hline
$(4,2,1)$  & $(\{1\},\{4\})$ & $(\{1\},\{3\})$ & $s_3 s_2 s_1$ & $(4,3,2,1)$\\
\hline
$(2,1,3)$  & $(\{1\},\{3\})$ & $(\{1\},\{2\})$ & $s_2 s_1$ & $(3,2,1)$\\
\hline
$(1,3)(2,4)$ & $(\{1,2\},\{3,4\})$ & $(\{1,2\},\{2,3\})$ & $(s_2 s_1) (s_3 s_2)$ & $(1,3)(2,4)$\\
\hline
$(1,2)(3,4)$  & $(\{1,3\},\{2,4\})$ & $(\{1,3\},\{1,3\})$ & $(s_1) (s_3)$ & $(1,2)(3,4)$\\
\hline
$(1,3,4,2)$  & $(\{1,3\},\{3,4\})$ & $(\{1,3\},\{2,3\})$ & $(s_2 s_1) (s_3)$ & $(1,3,4,2)$\\
\thickhline

\end{tabular}
}
\caption{The bijection $x\mapsto \varphi(x)$ for $c=(1,3,4,2)$.}
\label{fig:thebijection}
\end{figure}
\end{center}

\end{exple}
\begin{rmq}\label{rmq:bijcasc}
In case $c=c_{\mathrm{lin}}$, since $\inv=\mathrm{id}$ (see Remark \ref{rmq:triv}) we have
$$(D_x^{c}, U_x^{c})=(J_{\varphi(x)}, I_{\varphi(x)}[1])$$
for all $x\in\NC(\Sn,c)$. The bijection $\varphi$ is hence much easier to compute in that case.
\end{rmq}

\section{{Standard forms and exotic orders on noncrossing partitions}
}

In this section, we give an (hopefully self-contained) exposition of some results of \cite{GobWil} on exotic orders on noncrossing partitions. More precisely, for any $c\in\Std(\Sn)$ we define a partial order on $\NC(\Sn,c)$, endowing $\NC(\Sn,c)$ with the structure of a graded distributive lattice. If $c=c_{\mathrm{lin}}$ then this poset is the restriction to $\NC(\Sn,c)$ of the Bruhat order on $\Sn$ which unexpectedly yields a lattice structure --- for other Coxeter elements in general the order is distinct from the restriction of Bruhat order. The isomorphy type of the lattice is the same for all $c\in\Std(\Sn)$ and is that of the lattice of nonnesting partitions (or order ideals in the type $A_n$ root poset --- see \cite[Section 5.1]{Arm} for more on this lattice). To this end we first recall from \cite{GobWil} the definition of standard forms of noncrossing partitions, which are some canonical (possibly non-reduced) $\S$-words for them.

Linear extensions of these orders, together with the bijection from Theorem~\ref{thm:bijecgenerales}, will provide triangular base changes in the classical Temperley-Lieb algebra between bases indexed respectively by noncrossing partitions and fully commutative permutations.

\subsection{Standard forms}\label{sub:std}  

Note that if $x\in\NC(\Sn, c)$ is a cycle, there exists a unique cycle $y\in\NC(\Sn, c_{\mathrm{lin}})$ with $\mathrm{supp}(x)=\mathrm{supp}(y)$. We first take a $\T$-reduced expression of the cycle $x$ having as letters the transpositions $(d_i, d_{i+1})$, $i=1,\dots,k-1$ where $\mathrm{supp}(x)=\{d_1 <d_2 <\cdots < d_k \}$. Such a reduced expression always exists and is not unique in general, but any two of these reduced expressions are equal up to commutation. The transpositions occurring in a reduced expression as above are exactly the canonical Coxeter generators of Dyer \cite{Dyerrefn} of the parabolic subgroup of $\Sn$ defined by $x$. We call a $\T$-reduced expression as above \defn{distinguished}. 

\begin{rmq}
In fact, a $\T$-reduced expression as above can also be defined as follows: chose an $\S$-reduced expression of $c$. Take the $c$-sorting word for $w_0$ with respect to that choice of reduced expression of $c$ (see \cite{Read}). Then it follows from \cite{ABW} that the corresponding reflection order has a unique subword which is a $\T$-reduced expression of $x$ and the letters occurring in it are exactly the canonical generators of the parabolic subgroup defined by $x$.
\end{rmq}

\begin{definition}
Take a distinguished expression of a cycle $x\in\NC(\Sn, c)$. Replace any transposition $(i, j)$ with $i<j$ occurring in it by the $\S$-word $s_{j-1} s_{j-2}\cdots s_i\cdots s_{j-1}$ (which we call a \defn{syllable}) to get an $\S$-word representing $x$. Call this word a \defn{standard form} of $x$. The simple transposition $s_i$ is the \defn{center} of the above syllable and we say that the transposition $s_{j-1}$ occurs \defn{at the top} of the syllable. The center of a syllable splits the syllable: the letters occurring on the left (resp. on the right) of the center form the \defn{left part} (resp. \defn{right part}) of the syllable. 
\end{definition}

\begin{exple}
Let $x=(1,8,5,4)\in \NC(\mathfrak{S}_8, c)$ where $c=(1,3,6,7,8,5,4,2)$. A distinguished expression of $x$ is given by $(8,5)(5,4)(4,1)$. Then a standard form of $x$ is $$(s_7 s_6 s_5 s_6 s_7)(s_4)(s_3 s_2 s_1 s_2 s_3).$$ The brackets indicate the syllables. 
\end{exple}

\begin{rmq}\label{rmq:lettre1fois}
Note that if $s_k\in \S$ occurs in a standard form of a cycle $x\in\NC(\Sn,c)$, then it occurs in exactly one syllable. Indeed, it follows from the above construction that distinct syllables have disjoint $\S$-support. 
\end{rmq}

A standard form of a cycle $x$ is always an $\S$-reduced expression of $x$. The aim now is to define standard forms of abritrary noncrossing partitions by concatenating standard forms of the various cycles of the noncrossing partition. For some proofs in the following sections, we need to fix a specific rather subtle order on the cycles, equivalently on the polygons. Namely, given $x\in\NC(\Sn, c)$, we define a total order $<$ on $\Pol(x)=\{P_1,\dots, P_k\}$ as follows: set $n_i:=\min P_i$ for all $i=1,\dots,k$. Then $P_i < P_j$ if and only if when going along the circle in counterclockwise order starting at the point $n+1$, the point $n_i$ is met before the point $n_j$. An example is given in Figure~\ref{figure:ordre}.

\begin{definition}
Given $x\in\NC(\Sn,c)$, $\Pol(x)=\{P_1, \dots, P_k\}$ where $i<j$ if and only if $P_i< P_j$ and denoting by $c_i$ the cycle corresponding to $P_i$, we define a \defn{standard form} of $x$ to be any $\S$-word for $x$ obtained by concatenating standard forms of $c_1, c_2,\dots, c_k$ (in this order). We denote a standard form of $x$ by $m_x^c$. It is unique up to commutation of adjacent syllables representing commuting reflections of the same polygon.
\end{definition}
The following is clear by construction
\begin{lemma}\label{std}
Let $x\in\NC(\W,c)$. Then 
\begin{enumerate}
\item A letter $s_k$ is the center of a syllable of $m_x^c$ if and only if $k$ is a non-terminal index of a polygon of $x$. 
\item A letter $s_k$ is at the top of a syllable of $m_x^c$ if and only if $k+1$ is a non-initial index of a polygon of $x$. 
\item A letter $s_k$ is the center (or the top) of at most one syllable of $m_x^c$.
\end{enumerate}
\end{lemma}

In general a standard form of a noncrossing partition which is not a cycle is not $\S$-reduced (see Example~\ref{explecc} below) but it is always reduced when $c=c_{\mathrm{lin}}$ (by \cite[Lemma 2.8]{Gob}). 

\begin{figure}[h!]
\centering
{\begin{minipage}{8cm}
\psscalebox{1.4}{\begin{pspicture}(-2,-1.92)(2.5,1.92)
\pscircle[linecolor=gray, linewidth=0.2pt](0,0){1.8}
\pscurve[linecolor=red, linewidth=1.5pt](1.2237,1.32)(0, 1.1)(-1.44,1.08)
\pscurve[linecolor=red, linewidth=1.5pt](1.2237,1.32)(1.3,0.84)(1.7636,0.36)
\pscurve[linecolor=red, linewidth=1.5pt](-1.44,1.08)(0, 0.8)(1.7636,0.36)

\psline[linecolor=red, linewidth=1.5pt](-1.59197,0.84)(1.796,-0.12)
\psline[linecolor=red, linewidth=1.5pt](-1.59197,0.84)(1.796,0.12)
\pscurve[linecolor=red, linewidth=1.5pt](1.796, -0.12)(1.5, 0.08)(1.796, 0.12)

\psline[linecolor=red, linewidth=1.5pt](1.7636,-0.36)(-1.697,-0.6)
\psline[linecolor=red, linewidth=1.5pt](1.7636,-0.36)(-1.44,-1.08)
\pscurve[linecolor=red, linewidth=1.5pt](-1.697,-0.6)(-1.44,-0.8)(-1.44,-1.08)

\psline[linecolor=red, linewidth=1.5pt](1.59197,-0.84)(-1.2237,-1.32)

\rput(0.6, 0.9){\tiny $P_1$}
\rput(-0.6, 0.35){\tiny $P_2$}
\rput(-0.6, -0.7){\tiny $P_3$}
\rput(0.6, -1.2){\tiny $P_4$}

\psdots(0,1.8)(0.897,1.56)(-1.44,1.08)(1.697,0.6)(1.7636,0.36)(1.796,0.12)(1.796,-0.12)(1.7636,-0.36)(-1.697,-0.6)(1.59197,-0.84)(-1.44,-1.08)(-1.2237,-1.32)(-0.897,-1.56)(0,-1.8)

\psdots[linecolor=blue, fillcolor=blue](1.2237,1.32)(-1.59197,0.84)(1.7636,-0.36)(1.59197,-0.84)

\psline[linestyle=dotted, linewidth=0.4pt](-2,1.8)(2,1.8)
\psline[linestyle=dotted, linewidth=0.4pt](-2,1.56)(2,1.56)
\psline[linestyle=dotted, linewidth=0.4pt](-2,1.32)(2,1.32)
\psline[linestyle=dotted, linewidth=0.4pt](-2,1.08)(2,1.08)
\psline[linestyle=dotted, linewidth=0.4pt](-2,0.84)(2,0.84)
\psline[linestyle=dotted, linewidth=0.4pt](-2,0.6)(2,0.6)
\psline[linestyle=dotted, linewidth=0.4pt](-2,0.36)(2,0.36)
\psline[linestyle=dotted, linewidth=0.4pt](-2,0.12)(2,0.12)
\psline[linestyle=dotted, linewidth=0.4pt](-2,-1.8)(2,-1.8)
\psline[linestyle=dotted, linewidth=0.4pt](-2,-1.56)(2,-1.56)
\psline[linestyle=dotted, linewidth=0.4pt](-2,-1.32)(2,-1.32)
\psline[linestyle=dotted, linewidth=0.4pt](-2,-1.08)(2,-1.08)
\psline[linestyle=dotted, linewidth=0.4pt](-2,-0.84)(2,-0.84)
\psline[linestyle=dotted, linewidth=0.4pt](-2,-0.6)(2,-0.6)
\psline[linestyle=dotted, linewidth=0.4pt](-2,-0.36)(2,-0.36)
\psline[linestyle=dotted, linewidth=0.4pt](-2,-0.12)(2,-0.12)

\rput(2.2, 1.8){\tiny $1$}
\rput(2.2, 1.56){\tiny $2$}
\rput(2.2, 1.32){\tiny $3$}
\rput(2.2, 1.08){\tiny $4$}
\rput(2.2, 0.84){\tiny $5$}
\rput(2.2, 0.6){\tiny $6$}
\rput(2.2, 0.36){\tiny $7$}
\rput(2.2, 0.12){\tiny $8$}
\rput(2.2, -0.12){\tiny $9$}
\rput(2.2, -0.36){\tiny $10$}
\rput(2.2, -0.6){\tiny $11$}
\rput(2.2, -0.84){\tiny $12$}
\rput(2.2, -1.08){\tiny $13$}
\rput(2.2, -1.32){\tiny $14$}
\rput(2.2, -1.56){\tiny $15$}
\rput(2.2, -1.8){\tiny $16$}
\end{pspicture}}
 \end{minipage}
\begin{minipage}{5.5cm} 
\caption{}
\label{figure:ordre}
Example of the order $<$ on polygons. The initial indices of polygons are in blue. We have $$P_4< P_3< P_1< P_2.$$ \end{minipage}}
\end{figure}
The following is clear
\begin{lemma}\label{lem:centre1}
Let $x\in\NC(\Sn, c)$, $P, Q\in\Pol(x)$, $k\in\Vert(P)$. 
\begin{enumerate}
\item If $k\in L_c$ and $k$ is nested in $Q$, then $Q<P$,
\item If $k\in R_c$ and $k$ is nested in $Q$, then $Q>P$.
\end{enumerate}
\end{lemma}

\subsection{Exotic lattice structure on noncrossing partitions}\label{sub:lattice}

\begin{definition}\label{def:vertical}
Let $x\in\NC(\Sn,c)$. The \defn{vertical vector} of $x$ is the $n$-tuple of nonnegative integers $(x_i^c)_{i=1}^n$, where $x_i^c$ is the number of occurrences of the letter $s_i$ in a standard form $m_x^c$ of $x$. It is well-defined since standard forms are unique up to commutation of syllables. 
\end{definition}
The terminology comes from a realization of the standard forms as order ideals in a labeled root poset (see \cite[Section 5.2]{GobWil}).
\begin{exple}\label{explecc}
Let $c=(1,3,4,2)$, then $x=(1,3)(2,4)\in\NC(\Sn,c)$ and a standard form of $x$ is given by $(s_2 s_1 s_2)(s_3 s_2 s_3)$. Note that it is not an $\S$-reduced expression of $x$ since $\ell(x)=4$. The corresponding vertical vector is given by $(1,3,2)$. 
\end{exple} 

\begin{rmq}
The set of vertical vectors obtained from all the noncrossing partitions in $\NC(\Sn,c)$ is characterized in \cite[Proposition 6.4]{GobWil}. It is independent of the choice of the Coxeter element $c$ by \cite[Theorem 1.3]{GobWil}. We denote it by $\mathcal{V}_n$.
\end{rmq}
We order $\NC(\Sn, c)$ by the componentwise order $\leqv$ on the corresponding vertical vectors, that is, $$x \leqv y~\Leftrightarrow~x_i^c\leq y_i^c~\forall i=1,\dots,n.$$ This endows $\NC(\Sn,c)$ with a graded distributive lattice structure, isomorphic to the lattice of nonnesting partitions (see \cite[Section 7 and Theorems 1.2 and 1.3]{GobWil}). See Figure~\ref{figure:a3} for an example.

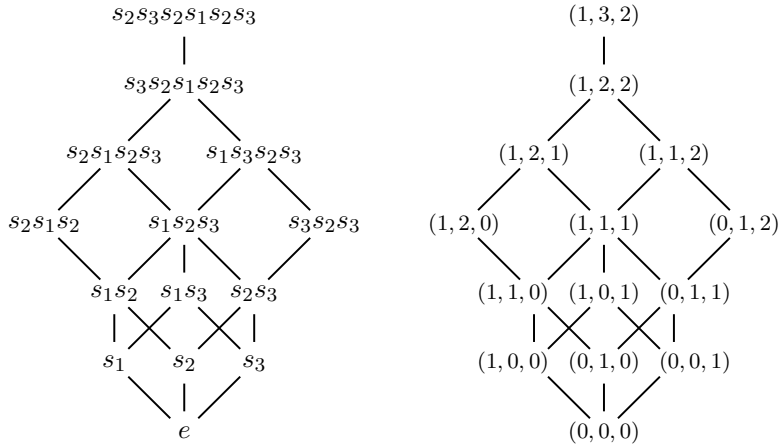
\begin{figure}[htbp]
\psscalebox{0.92}{\begin{pspicture}(0,0)(10,6.2)
\psline(1.8,0.2)(1.2,0.8)
\psline(1.2,1.2)(1.8,1.8)
\psline(2.8,1.2)(2.2,1.8)
\psline(2.8,2.2)(2.2,2.8)
\psline(2.8,4.2)(2.2,4.8)
\psline(0.8,2.2)(0.2,2.8)
\psline(0.8,3.8)(0.2,3.2)
\psline(3.2,2.2)(3.8,2.8)
\psline(3.8,3.2)(3.2,3.8)
\psline(2,0.3)(2,0.7)
\psline(2,2.3)(2,2.7)
\psline(2,5.3)(2,5.7)
\psline(1,1.3)(1,1.7)
\psline(3,1.3)(3,1.7)
\psline(2.2,0.2)(2.8,0.8)
\psline(1.8,1.2)(1.2,1.8)
\psline(1.8,3.2)(1.2,3.8)
\psline(2.2,1.2)(2.8,1.8)
\psline(2.2,3.2)(2.8,3.8)
\psline(1.2,2.2)(1.8,2.8)
\psline(1.2,4.2)(1.8,4.8)
\rput(2,0){$e$}
\rput(1,1){$s_1$}
\rput(2,1){$s_2$}
\rput(3,1){$s_3$}
\rput(1,2){$s_1 s_2$}
\rput(2,2){$s_1 s_3$}
\rput(3,2){$s_2 s_3$}
\rput(2,3){$s_1 s_2 s_3$}
\rput(0,3){$s_2 s_1 s_2$}
\rput(4,3){$s_3s_2s_3$}
\rput(1,4){$s_2s_1s_2s_3$}
\rput(3,4){$s_1 s_3s_2s_3$}
\rput(2,5){$s_3s_2s_1s_2s_3$}
\rput(2,6){$s_2s_3s_2s_1s_2s_3$}
\psline(7.8,0.2)(7.2,0.8)
\psline(7.2,1.2)(7.8,1.8)
\psline(8.8,1.2)(8.2,1.8)
\psline(8.8,2.2)(8.2,2.8)
\psline(8.8,4.2)(8.2,4.8)
\psline(6.8,2.2)(6.2,2.8)
\psline(8.8,3.8)(8.2,3.2)
\psline(9.2,2.2)(9.8,2.8)
\psline(9.8,3.2)(9.2,3.8)
\psline(8,0.3)(8,0.7)
\psline(8,2.3)(8,2.7)
\psline(7,1.3)(7,1.7)
\psline(8.2,0.2)(8.8,0.8)
\psline(7.8,1.2)(7.2,1.8)
\psline(7.8,3.2)(7.2,3.8)
\psline(8.2,1.2)(8.8,1.8)
\psline(6.2,3.2)(6.8,3.8)
\psline(7.2,2.2)(7.8,2.8)
\psline(7.2,4.2)(7.8,4.8)
\psline(9,1.3)(9,1.7)
\psline(8,5.3)(8,5.7)
\rput(8,0){\footnotesize $(0,0,0)$}
\rput(6.7,1){\footnotesize $(1,0,0)$}
\rput(8,1){\footnotesize $(0,1,0)$}
\rput(9.3,1){\footnotesize $(0,0,1)$}
\rput(6.7,2){\footnotesize $(1,1,0)$}
\rput(8,2){\footnotesize $(1,0,1)$}
\rput(9.3,2){\footnotesize $(0,1,1)$}
\rput(8,3){\footnotesize $(1,1,1)$}
\rput(6,3){\footnotesize $(1,2,0)$}
\rput(10,3){\footnotesize $(0,1,2)$}
\rput(7,4){\footnotesize $(1,2,1)$}
\rput(9,4){\footnotesize $(1,1,2)$}
\rput(8,5){\footnotesize $(1,2,2)$}
\rput(8,6){\footnotesize $(1,3,2)$}
\end{pspicture}}
\caption{The elements of $\NC(\mathfrak{S}_4,s_1s_2s_3)$ and the corresponding set $\mathcal{V}_3$ of vertical vectors, endowed with the partial order $\leqv$. Note that the lattice on the left coincides with the restriction of Bruhat order to $\NC(\mathfrak{S}_4, s_1s_2s_3)$.}
\label{figure:a3}
\end{figure}
\begin{rmq}
Note that the above lattice structure already occurred in the framework of Temperley-Lieb algebras in \cite{GS}\footnote{The author thanks Frédéric Chapoton for pointing out this reference to his attention.}. The lattice there comes from a partial order on $\F(\Sn)$ (which index a basis of the Temperley-Lieb algebra, see Section~\ref{tl:diagram} below). Under our bijection $\varphi$ (which we will show in Section~\ref{sec:triang} to come from a triangular base change in the Temperley-Lieb algebra), the lattice on $\leqv$ can be transformed into a lattice on $\F(\Sn)$ which is different from that of \cite{GS}: the fully commutative elements lie in different places on the two Hasse diagrams. We do not know whether there is a link between the results in~\cite{GS} and ours.
\end{rmq}

\section{Dual braid monoids and lifts of noncrossing partitions}

\subsection{Artin braid group and dual braid monoids}
The reference for this material is mainly \cite{Dual}. Let $(\mathcal{W}, \mathcal{S})$ be a finite Coxeter system. Here we use the notation from the beginning of Subsection \ref{sec:nc_algebraic}. Write $B=B(\mathcal{W},\mathcal{S})$ for the \defn{Artin-Tits group} associated to $(\W,\S)$. It is the group generated by a copy $\mathbf{S}:=\{\mathbf{s}~|~s\in \S\}$ (called the \defn{Artin generators}) of $\S$ subject only to the braid relations, that is, $$B=\langle\mathbf{S}\mid \underbrace{\mathbf{s}\mathbf{t}\cdots}_{m_{st}}=\underbrace
{\mathbf{t}\mathbf{s}\cdots}_{m_{ts}} \text{ for } \mathbf{s},
\mathbf{t}\in\mathbf{S}\rangle.$$
We call a word in $\mathbf{S}\cup\mathbf{S}^{-1}$ a \defn{braid word}. We denote by $p:B\longrightarrow\mathcal{W}$ the canonical surjection defined on the generators by $p(\mathbf{s})=s$. 

Let $c\in\Std(\W)$. Following an idea of Birman, Ko and Lee~\cite{BKL}, Bessis \cite[Definition 2.1.1]{Dual} defined the \defn{dual braid monoid} $B_c^*$ associated to $(\mathcal{W}, \mathcal{S}, c)$ as the monoid generated by a copy $\{i_c(t)~|~t\in\mathcal{T}\}$ of $\mathcal{T}=\bigcup_{w\in\W} w\S w^{-1}$ with relations $$i_c(t) i_c(t')=i_c(tt't)i_c(t),~\text{if }tt'\in\NC(\W,c).$$
The relations above are the \defn{dual braid relations}. There is an embedding $B_c^*\hookrightarrow B$ and the group of fractions of $B_c^*$ is isomorphic to $B$ (\cite[Theorem 2.2.5]{Dual}) and the left-divisibility relation is a partial order on $B_c^*$ endowing it with a lattice structure (in fact, $B_c^*$ is a \defn{Garside monoid} \cite[Section 2.3]{Dual}, which implies some of these properties; we do not introduce the machinery of Garside theory here \cite{DDGKM} since it is not required for our purposes). For any $x\in\NC(\W,c)$ with a $\T$-reduced expression $t_1\cdots t_k$ the product $$i_c(t_1)\cdots i_c(t_k)$$ in $B_c^*$ is independent of the chosen reduced expression (see \cite[Section 1.6]{Dual}). We denote it by $i_c(x)$. The elements of the form $i_c(x)$ where $x\in\NC(\W,c)$ form the set of \defn{simple elements} of $B_c^*$ or \defn{simples} or which we denote by $\bD_c$ (such a set is associated to any Garside monoid, see \cite{DDGKM}). To emphasize that we work with the dual Garside structure we will call them \defn{simple dual braids}. The restriction of the left divisibility order on $B_c^*$ endows $\bD_c$ with a lattice structure, and the canonical bijection $$(\NC(\W,c), \leq_\T) \overset{\sim}{\longrightarrow} \bD_c, ~x\mapsto i_c(x)$$ is an isomorphism of posets. We will also denote by $i_c(x)$ the image of it in $B$ under the canonical embedding. 

The embedding $B_c^*\hookrightarrow B$ maps $i_c(s)$ to $\mathbf{s}$ for any $s\in S$. Since $B_c^*$ is generated by a copy of $\T$ one would also like to express the elements $i_c(t)$ (for $t\in\T\backslash\S$) as a word in $\mathbf{S}\cup\mathbf{S}^{-1}$. A general formula for the lifts of reflections is given in \cite[Proposition 3.13]{DG}, but it might not give a braid word of smallest possible length --- for our purposes we would like such a word, at least for the lifts of the reflections. In type $A_n$ we address this question in Lemma~\ref{lem:atome} below.

\begin{exple}\label{ex:sa}
Let $\W=\mathfrak{S}_3$. Let $c=s_1 s_2=(1,2,3)$. Then $\lt(c)=2$ and the poset $\NC(\mathfrak{S}_3, c)$ has $5$ elements. One has $i_c(s_1)=\mathbf{s}_1$, $i_c(s_2)=\mathbf{s}_2$. There is a dual braid relation $$i_c(s_1) i_c (s_2)= i_c (s_2) i_c(s_2 s_1 s_2),$$
hence inside $B$ we have $$i_c(s_2 s_1 s_2)= \mathbf{s}_2^{-1} \mathbf{s}_1 \mathbf{s}_2=\mathbf{s}_1 \mathbf{s}_2 \mathbf{s}_1^{-1}.$$
\end{exple}

\subsection{Simple dual braids of type $A_n$}
From now on, $(\W,\S)$ is of type $A_n$, $c\in\Std(\W)$ with the conventions already introduced. In that case $B$ is the Artin braid group on $n+1$ strands (see \cite{KaTu}). The aim of this section is to choose specific braid words to represent the elements of $\bD_c$ inside $B$ as in Example \ref{ex:sa}. We first treat the case of the reflections.

\begin{lemma}\label{lem:atome}
Let $i,j\in[n+1]$, $j< i$. Let $t\in\mathcal{T}$ be the transposition $(j, i)$. Then $i_{c}(t)$ is represented in $B$ by the braid word $$\mathbf{s}_{[i,j]}^c:=\mathbf{s}_{i-1}^{\varepsilon_{i-1}}\cdots \mathbf{s}_{j+1}^{\varepsilon_{j+1}}\mathbf{s}_j\mathbf{s}_{j+1}^{-\varepsilon_{j+1}}\cdots \mathbf{s}_{i-1}^{-\varepsilon_{i-1}},$$
where for all $k\in [j+1,i-1]$, $\varepsilon_k=1$ if $k\in L_{c}$, $\varepsilon_k=-1$ if $k\in R_{c}$.  
\end{lemma}
Note that in general $\mathbf{s}_{[i,j]}^c\neq \mathbf{s}_{[i,j]}^{c'}$ for $c\neq c'$.
\begin{proof}
We argue by induction on $i-j$. If $i-j=1$, then $(j,i)=(j,j+1)=s_j$ and $\mathbf{s}_j$ is equal to the image of $i_{c}(s_j)$ in $B$. Assume that $i-j>1$. If $i-1\in R_{c}$, then the triple $((j,i-1), s_{i-1}, (j,i))$ is a $c$-admissible triangle (see Definition~\ref{def:admiss}). Hence we have the dual braid relation $$i_{c}((j,i-1))i_{c}(s_{i-1})=i_{c}(s_{i-1})i_{c}((j, i)).$$
By induction we have $i_c((j,i-1))=\mathbf{s}_{i-2}^{\varepsilon_{i-2}}\cdots \mathbf{s}_{j+1}^{\varepsilon_{j+1}}\mathbf{s}_j\mathbf{s}_{j+1}^{-\varepsilon_{j+1}}\cdots \mathbf{s}_{i-2}^{-\varepsilon_{i-2}}$
where for all $k\in[j+1,i-2]$, $\varepsilon_k=1$ if $k\in L_{c}$, $\varepsilon_k=-1$ if $k\in R_{c}$. Hence $i_c((j,i))$ is represented in $B$ by the braid word
$$\mathbf{s}_{i-1}^{-1} (\mathbf{s}_{i-2}^{\varepsilon_{i-2}}\cdots \mathbf{s}_{j+1}^{\varepsilon_{j+1}}\mathbf{s}_j\mathbf{s}_{j+1}^{-\varepsilon_{j+1}}\cdots \mathbf{s}_{i-2}^{-\varepsilon_{i-2}}) \mathbf{s}_{i-1}$$
which gives the claimed formula in that case. 

If $i-1\in L_{c}$, then the triple $((j,i),s_{i-1},(j,i-1))$ is a $c$-admissible triangle. Hence we have the relation $$i_{c}(s_{i-1})i_{c}((j,i-1))=i_{c}((j, i))i_{c}(s_{i-1})$$ from which we similarly derive the claimed formula.
\end{proof}
In the well-known representation of $B$ by Artin braids (see \cite{KaTu}), with the convention that $\mathbf{s}_i$ has the $i$\ts{th} strand below the $(i+1)$\ts{th}, it follows from Lemma~\ref{lem:atome} that the braid $\mathbf{s}_{[i,j]}$ is represented by a crossing of the $i$\ts{th} and $j$\ts{th} strand, with the $j$\ts{th} strand below; the other strands are unbraided, with a $k$\ts{th} strand above the crossing if $k\in R_c$ and below if $k\in L_c$ (we concatenate Artin braids from top to bottom). Hence the lift of $t$ in $B$ can easily be described as an Artin braid --- an Example is given in Figure \ref{fig:liftref}: replacing the edge representing the reflection in the circle representation with $c$-labeling by a loop with a counterclockwise orientation gives the Artin braid viewed in a cylinder.  

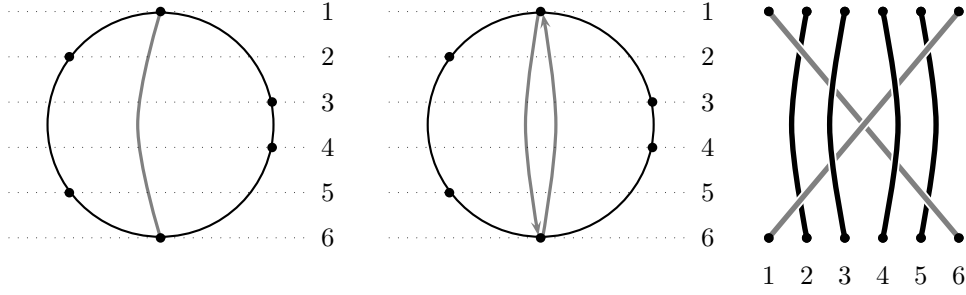
\begin{figure}[h!]
\psscalebox{1}{
\begin{pspicture}(3.9,-1)(10.9,3)
\pscircle(3,1.5){1.5}
\pscurve[linecolor=gray, linewidth=1.2pt](3,3)(2.7, 1.5)(3, 0)

\psline[linestyle=dotted, linewidth=0.4pt](1,3)(5,3)
\psline[linestyle=dotted, linewidth=0.4pt](1,2.4)(5,2.4)
\psline[linestyle=dotted, linewidth=0.4pt](1,1.8)(5,1.8)
\psline[linestyle=dotted, linewidth=0.4pt](1,1.2)(5,1.2)
\psline[linestyle=dotted, linewidth=0.4pt](1,0.6)(5,0.6)
\psline[linestyle=dotted, linewidth=0.4pt](1,0)(5,0)

\psdots(3,3)(1.8, 2.4)(4.4696, 1.8)(4.4696, 1.2)(1.8, 0.6)(3,0)
\rput(5.2,3){\small $1$}
\rput(5.2,2.4){\small $2$}
\rput(5.2,1.8){\small $3$}
\rput(5.2,1.2){\small $4$}
\rput(5.2,0.6){\small $5$}
\rput(5.2,0){\small $6$}

\pscircle(8,1.5){1.5}
\pscurve[linecolor=gray, linewidth=1.2pt]{->}(7.97,2.97)(7.8, 1.5)(7.97, 0.03)
\pscurve[linecolor=gray, linewidth=1.2pt]{->}(8.03,0.03)(8.2, 1.5)(8.03,2.97)

\psline[linestyle=dotted, linewidth=0.4pt](6,3)(10,3)
\psline[linestyle=dotted, linewidth=0.4pt](6,2.4)(10,2.4)
\psline[linestyle=dotted, linewidth=0.4pt](6,1.8)(10,1.8)
\psline[linestyle=dotted, linewidth=0.4pt](6,1.2)(10,1.2)
\psline[linestyle=dotted, linewidth=0.4pt](6,0.6)(10,0.6)
\psline[linestyle=dotted, linewidth=0.4pt](6,0)(10,0)

\psdots(8,3)(6.8, 2.4)(9.4696, 1.8)(9.4696, 1.2)(6.8, 0.6)(8,0)
\rput(10.2,3){\small $1$}
\rput(10.2,2.4){\small $2$}
\rput(10.2,1.8){\small $3$}
\rput(10.2,1.2){\small $4$}
\rput(10.2,0.6){\small $5$}
\rput(10.2,0){\small $6$}

\psdots(11,3)(11.5,3)(12,3)(12.5,3)(13,3)(13.5,3)

\psdots(11,0)(11.5,0)(12,0)(12.5,0)(13,0)(13.5,0)

\pscurve[linewidth=2pt](11.5,0)(11.3,1.5)(11.5,3)
\pscurve[linewidth=2pt](13,0)(13.2,1.5)(13,3)
\psline[linewidth=3.4pt, linecolor=white](11,3)(13.5,0)
\psline[linewidth=2pt, linecolor=gray](11,3)(13.5,0)
\psline[linewidth=3.4pt, linecolor=white](11,0)(13.5,3)
\psline[linewidth=2pt, linecolor=gray](11,0)(13.5,3)
\pscurve[linewidth=3.4pt, linecolor=white](12,3)(11.8,1.5)(12,0)
\pscurve[linewidth=2pt](12,3)(11.8,1.5)(12,0)
\pscurve[linewidth=3.4pt, linecolor=white](12.5,3)(12.7,1.5)(12.5,0)
\pscurve[linewidth=2pt](12.5,3)(12.7,1.5)(12.5,0)

\psdots(11,3)(11.5,3)(12,3)(12.5,3)(13,3)(13.5,3)

\psdots(11,0)(11.5,0)(12,0)(12.5,0)(13,0)(13.5,0)

\rput(11,-0.5){\small $1$}
\rput(11.5,-0.5){\small $2$}
\rput(12,-0.5){\small $3$}
\rput(12.5,-0.5){\small $4$}
\rput(13,-0.5){\small $5$}
\rput(13.5,-0.5){\small $6$}

\end{pspicture}}
\caption{Lift $i_c(t)$ of the transposition $t=(1,6)\in\NC(\mathfrak{S}_6,c)$, where $c=(1,3,4,6,5,2)$.}
\label{fig:liftref}
\end{figure}

The aim is now to represent any element $i_c(x)$, $x\in\NC(\Sn,c)$ by a braid word as we did for $i_c(t)$ above. To this end, we simply lift a standard form $m_x^c$ of $x$ by lifting any syllable of it using Lemma~\ref{lem:atome}. Since a standard form is built starting with a $\T$-reduced expression such a process yields a word representing $i_c(x)$ in $B$ by definition of $B_c^{*}$. We denote a lifted standard form of $x$ by $\bm_x^c$. We also call $\bm_x^c$ a \defn{standard form} of $i_c(x)$.

\begin{exple}
Let $c=(1,2,5,4,3)\in\Std(\mathfrak{S}_5)$, let $x=(1,3)(2,5,4)$. A standard form $m_{c_1}^c$ of $c_1:=(1,3)$ is given by $s_2 s_1 s_2$ while writing $c_2:=(2,5,4)$ as $(4,5)(2,4)$ we get a standard form $m_{c_2}^c= s_4 (s_3 s_2 s_3)$. We have $P<Q$ where $P,Q\in\Pol(x)$, $\Vert(P)=\{2,4,5\}$ are such that $\Vert(Q)=\{ 1,3\}$, hence a standard form of $x$ is $$m_x^c=m_{c_2}^c m_{c_1}^c=s_4 (s_3 s_2 s_3)(s_2 s_1 s_2).$$
Note that $m_x^c$ is not $\S$-reduced since $x= s_4 s_2 s_3 s_1 s_2$. Lifting each syllable using Lemma~\ref{lem:atome} we get a braid word for $i_c(x)$ given by 
$\bm_x^c=\mathbf{s}_4 (\mathbf{s}_3 \mathbf{s}_2 \mathbf{s}_3^{-1})(\mathbf{s}_2^{-1} \mathbf{s}_1 \mathbf{s}_2).$

\end{exple}

\begin{rmq}
Simple dual braids of type $A_n$ turn out to be so-called \defn{Mikado braids} (see~\cite[Section 5.3]{DG}) which in particular implies that any element $i_c(x)$ can be represented by a braid word of length $\ls(x)$. The standard form of $i_c(x)$ or $x$ is in general of bigger length than $\ls(x)$ as we already pointed out above. However the standard forms will allow us to describe orders giving triangular bases changes, which are precisely those introduced in Subsection \ref{sub:lattice} for which the use of $m_x^c$ and the fact that it is a possibly non-reduced word is crucial. 
\end{rmq}

\section{Classical Temperley-Lieb algebra}

The aim of this section is to introduce the classical Temperley-Lieb algebra and its diagram basis indexed by fully commutative permutations. There is a multiplicative homomorphism from the type $A_n$ Artin braid group to this algebra, and we also give some facts about images of simple dual braids inside the Temperley-Lieb algebra. 

\subsection{The Temperley-Lieb algebra and its diagram basis}\label{tl:diagram}

Let $n\geq 1$.
\begin{definition}
The (classical) \defn{Temperley-Lieb algebra} $\tl$ is the associative, unital $\A$-algebra with generators $b_1,\dots, b_n$ and relations 
\begin{eqnarray}
 b_j b_i b_j =b_j~\text{if $|i-j|=1$},\\
 b_i b_j= b_j b_i~\text{if $|i-j|>1$},\\
 b_i^2=(v+v^{-1}) b_i.
\end{eqnarray}
\end{definition}

  \begin{proposition}[Jones, \cite{JonesSub}]\label{prop:jones} 
Let $w\in \F(\Sn)$. To any $\S$-reduced expression $s_{i_1} s_{i_2}\cdots s_{i_k}$ of $w$, we associate the element $b_{i_1} b_{i_2}\cdots b_{i_k}$ of $\tl$.
  \begin{enumerate}
  \item The product $b_{i_1} b_{i_2}\cdots b_{i_k}$ is independent of the choice of the $\S$-reduced expression of $w$ and is therefore denoted by $b_w$.
  \item The set $\{b_w\}_{w\in\F(\Sn)}$ is a basis of $\tl$ as a $\A$-module, in particular $\tl$ is a free $\A$-module of rank $C_{n+1}$.
  \item For any sequence $j_1 j_2\cdots j_m$ of integers in $[n]$, there exists a unique pair $(x,k)\in\F(\Sn)\times\mathbb{Z}_{\geq 0}$ such that
  $$b_{j_1} b_{j_2}\cdots b_{j_m}=(v+v^{-1})^k b_x.$$
  \end{enumerate}
  \end{proposition}
  
\begin{definition}
The basis $\{b_w\}_{w\in\F(\Sn)}$ will be called the \defn{diagram basis} (because it has a diagrammatic version, see~\cite{Kau}). We do not introduce the diagrammatics here since it is not required for our purposes. 
\end{definition}

\subsection{Images of simple dual braids inside the Temperley-Lieb algebra}

There is a well-known multiplicative homomorphism $\omega:\Bni\rightarrow \tl$ mapping any Artin generator $\mathbf{s}\in\mathbf{S}$ to $v^{-1}-b_i$ and its inverse $\mathbf{s}^{-1}$ to $v-b_i$. This can be seen for instance by checking that the elements $t_i:=v^{-1}-b_i\in\tl$ satisfy the braid relations, which is an easy computation using the defining relations of $\tl$. For $c\in\Std(\Sn)$ and $x\in\NC(\Sn,c)$ set $R_x^c:=\omega(i_c(x))\in\tl$. The set $\{ R_x^c\}_{x\in \NC(\Sn,c)}$ therefore consists of the images of the simple elements of $B_c^*$ in $\tl$. Since it is indexed by noncrossing partitions which have Catalan enumeration, one can wonder whether $\{ R_x\}$ gives a basis of $\tl$. The answer is given by

\begin{theorem}[{Zinno~\cite[Theorem 2]{Z}, Vincenti~\cite[Corollaire 5.2.9]{Vincenti}}]\label{thm:z}
The elements $\{R_x^c\}_{x\in\NC(\Sn,c)}$ form a $\mathbb{Z}[v, v^{-1}]$-basis of $\tl$. 
\end{theorem}

\begin{rmq} Zinno proved the theorem above for $c=c_{\mathrm{lin}}$ (he works with slightly different conventions on the parameter but his theorem can be reformulated as above, see~\cite[Section 8.5]{DG}). Vincenti generalized Zinno's result to arbitrary standard Coxeter elements. 
\end{rmq}

\section{Triangular base changes}\label{sec:triang}
In this section, we show that if we order the basis $\{R_x^c\}$ (respectively $\{b_w\}$) by any linear extension of $\leqv$ (respectively by the order induced on $\F(\Sn)$ by taking the image under $\varphi$ of the chosen linear extension) then the base change between the two bases is upper triangular with invertible coefficients on the diagonal. Note that this gives in particular a new proof of Theorem~\ref{thm:z}, but our proof is more involved than the easiest known proof of Theorem~\ref{thm:z} (given by \cite{Vincenti} following an idea of \cite{LL}) since it is aimed at prooving that the base change is triangular and exhibiting explicitly the bijections induced on the underlying combinatorial objects indexing the bases.

Recall that $m_x^c$ is obtained as follows: write $x$ as a product of disjoint cycles, totally ordered as in Subsection~\ref{sub:std}. Replace each cycle by a distinguished expression of it. It gives a $\T$-reduced expression of $x$. Then replace each reflection in this expression by an $\S$-word for it (which we called a syllable). 

In the whole section, we use the notation $$m_x^c=w_1 w_2 \cdots w_\ell$$ to denote the $\S$-word $m_x^c$ as the concatenation of its various syllables $w_1, w_2,\dots, w_\ell$. 

\subsection{Properties of standard forms}

The aim of this subsection is to prove a few technical results on occurrences of letters in standard forms. 

\begin{rmq}\label{rmq:inversec}
Notice that inverting the Coxeter element $c$ reverses the order of the labeling of the circle, that is, we have $L_{c}=R_{c^{-1}}$ and $R_{c}=L_{c^{-1}}$. For $x\in\NC(\Sn,c)$ we have $x^{-1}\in\NC(\Sn,c^{-1})$ and the graphical representation of $x^{-1}$ is the image of the graphical representation of $x$ under the symmetry with vertical axis. The order on $\Pol(x)$ is the reverse order of the order on $\Pol(x^{-1})$. A simple transposition $s_k$ is a center of $m_x^{c}$ if and only if it is a center of $m_{x^{-1}}^{c^{-1}}$. An index $k$ is nested in (resp. is a vertex of) a polygon $P\in\Pol(x)$ if and only if it is nested in (resp. is a vertex of) a polygon $P'\in\Pol(x^{-1})$. Reversing a standard form for $x\in\NC(\Sn, c)$ we get a standard form for $x^{-1}\in\NC(\Sn, c^{-1})$
\end{rmq}

We can derive the following from Lemma~\ref{lem:centre1}

\begin{lemma}\label{lem:centre2}
Let $x\in \NC(\Sn,c)$. Let $s_k$ be the center of a syllable $w_i$ of $m_x^c$.
\begin{enumerate}
\item If $k\in L_c$ and $s_k$ occurs in a syllable $w_j$ with $i\neq j$, then $j <i$,
\item If $k\in R_c$ and $s_k$ occurs in a syllable $w_j$ with $i\neq j$, then $j>i$. 
\end{enumerate}
\end{lemma}
\begin{proof}
If $s_k$ is the center of a syllable of $x$, then $k$ is a (non-terminal) index of a polygon $P\in\Pol(x)$ (Lemma~\ref{std}). It follows from Remark~\ref{rmq:lettre1fois} that $w_j$ comes from a polygon $Q\neq P$ of $x$. Moreover, $s_k$ cannot be the center of $w_j$ by Lemma~\ref{std}. Hence $k$ is nested in $Q$. By Lemma~\ref{lem:centre1}, we conclude that $Q<P$, hence that $j<i$ since $w_j$ is a syllable of $Q$. The proof of the second statement follows by the first statement and Remark~\ref{rmq:inversec}. 
\end{proof}
Lemma~\ref{lem:centre2} gives the key property induced by the technical total order we put on $\Pol(x)$ which one should keep in mind: if $s_k$ is a center of a syllable of $m_x^c$ and $k\in L_c$ (respectively $k\in R_c$), then the last occurrence (resp. the first occurrence) of $s_k$ in $m_x^c$ is at the center of that syllable. Informations on letters appearing at the top can also be derived:

\begin{lemma}\label{lem:top}
Let $x\in\NC(\Sn,c)$. Let $s_k$ occur at the top of a syllable $w_i$ of $m_x^c$. Then $w_i$ is either the first or the last syllable of $m_x^c$ containing $s_k$. More precisely we have
\begin{enumerate}
\item If $s_k$ is a center of a syllable and $k\in R_c$ (resp. $k\in L_c$), then 
\begin{enumerate}
\item If $w_i\neq s_k$, then $w_i$ is the last (resp. first) syllable containing $s_k$,
\item If $w_i=s_k$, then $w_i$ is the first (resp. last) syllable containing $s_k$.
\end{enumerate}
\item If $s_k$ is not a center of $m_x^c$, then 
\begin{enumerate}
\item If $k+1\in R_c$, then $w_i$ is the first syllable containing $s_k$,
\item If $k+1\in L_c$, then $w_i$ is the last syllable containing $s_k$. 
\end{enumerate}
\end{enumerate}

\end{lemma}
\begin{proof}
Using Remark~\ref{rmq:inversec} we can assume that $k\in R_c$. First suppose that $s_k$ is a center. If $w_i=s_k$, then $s_k$ is the center of $w_i$ and therefore $w_i$ is the first syllable containing $s_k$ by Lemma~\ref{lem:centre2}. Hence assume that $w_i\neq s_k$. Note that $w_i$ is a syllable of a polygon $P\in\Pol(x)$ having $k+1$ as non-initial vertex (Lemma~\ref{std}). If $k+1\in R_c$, the only way to have $k$ as a non-terminal vertex and $k+1$ as a non-initial one with both of them in $R_c$ is in case both are vertices of the same polygon, implying $w_i=s_k$, a contradiction. Hence we can assume that $k+1\in L_c$. 

We now prove that if $k+1\in L_c$, then $w_i$ is the last syllable containing $s_k$ (without assuming $s_k$ to be a center, which therefore simultaneously proves $1(a)$ and $2(b)$). Consider the horizontal line $\mathcal{L}_{k+1}$ containing the point $k+1$ (see Remark~\ref{rmq:successifs}). Then any syllable $w_j$ containing the letter $s_k$ must be part of a polygon which meets $\mathcal{L}_{k+1}$. Any such polygon is below the polygon to which $w_i$ belongs in the order $<$, hence $w_i$ must be the last syllable containing $s_k$. 

It remains to treat the case where $s_k$ is not a center and $k+1\in R_c$. Again, the polygon $P$ containing $w_i$ has $k+1$ as vertex and any other polygon having a syllable with the letter $s_k$ must meet the line $\mathcal{L}_{k+1}$. But among all the polygons meeting $\mathcal{L}_{k+1}$, $P$ is the lowest one with respect to $<$. Hence $w_i$ is the first syllable containing $s_k$. 
\end{proof}
\begin{lemma}\label{lem:topcenter}
Let $x\in\NC(\Sn,c)$.  
\begin{enumerate}
\item Assume that $k, k+1\in R_{c}$ (or $k, k+1\in L_{c}$). If $s_k$ is the center of $w_i$ and at the top of $w_j$, then $i=j$ and therefore $w_i=s_k$.
\item Assume that $k\in L_{c}$, $k+1\in R_{c}$ (or $k\in R_{c}$, $k+1\in L_{c}$), $w_i=s_k$. Then $s_k$ does not occur in $w_j$ for $j\neq i$. Moreover, there is at most one occurrence of $s_{k+1}$ in $m_x^c$ (which must be as a center of a syllable of $P\in\Pol(x)$, where $P$ is the polygon containing $w_i$).
\end{enumerate}
\end{lemma}
\begin{figure}[h!]
\begin{center}
\psscalebox{0.9}{
\begin{pspicture}(4,-1.92)(2.5,1.92)
\pscircle[linecolor=gray, linewidth=0.2pt](0,0){1.8}

\psline[linecolor=gray, linewidth=1.5pt]{->}(0,1.8)(0,-1.8)

\psdots(0,1.8)(-0.897,1.56)(1.2237,1.32)(-1.44,1.08)(-1.59197,0.84)(1.697,0.6)(1.7636,0.36)(-1.796,0.12)(1.796,-0.12)(1.7636,-0.36)(-1.697,-0.6)(1.59197,-0.84)(-1.44,-1.08)(1.2237,-1.32)(-0.897,-1.56)(0,-1.8)
\pscurve[linestyle=dashed](1.796,-0.12)(0.5, 0.5)(-0.4, 0.3)(-1.697,-0.6)
\pscurve[linestyle=dashed](1.7636,-0.36)(1.3,0)(1.2237,1.32)
\psline[linestyle=dotted, linewidth=0.4pt](-2,1.8)(2,1.8)
\psline[linestyle=dotted, linewidth=0.4pt](-2,1.56)(2,1.56)
\psline[linestyle=dotted, linewidth=0.4pt](-2,1.32)(2,1.32)
\psline[linestyle=dotted, linewidth=0.4pt](-2,1.08)(2,1.08)
\psline[linestyle=dotted, linewidth=0.4pt](-2,0.84)(2,0.84)
\psline[linestyle=dotted, linewidth=0.4pt](-2,0.6)(2,0.6)
\psline[linestyle=dotted, linewidth=0.4pt](-2,0.36)(2,0.36)
\psline[linestyle=dotted, linewidth=0.4pt](-2,0.12)(2,0.12)
\psline[linestyle=dotted, linewidth=0.4pt](-2,-1.8)(2,-1.8)
\psline[linestyle=dotted, linewidth=0.4pt](-2,-1.56)(2,-1.56)
\psline[linestyle=dotted, linewidth=0.4pt](-2,-1.32)(2,-1.32)
\psline[linestyle=dotted, linewidth=0.4pt](-2,-1.08)(2,-1.08)
\psline[linestyle=dotted, linewidth=0.4pt](-2,-0.84)(2,-0.84)
\psline[linestyle=dotted, linewidth=0.4pt](-2,-0.6)(2,-0.6)
\psline[linestyle=dotted, linewidth=0.4pt](-2,-0.36)(2,-0.36)
\psline[linestyle=dotted, linewidth=0.4pt](-2,-0.12)(2,-0.12)

\rput(2.4, -0.6){\tiny $m$}
\rput(2.4, 1.32){\tiny $m'$}
\rput(2.4, -0.12){\tiny $k$}
\rput(2.4, -0.36){\tiny $k+1$}
\rput(2.4, 1.8){\tiny $1$}
\rput(2.4, -1.8){\tiny $n+1$}

\pscircle[linecolor=gray, linewidth=0.2pt](6,0){1.8}

\psline[linecolor=gray, linewidth=1.5pt]{->}(6,1.8)(6,-1.8)
\psline[linestyle=dashed](7.7636,0.36)(4.204,0.12)
\psdots(6,1.8)(5.103,1.56)(7.2237,1.32)(4.56,1.08)(4.40803,0.84)(7.697,0.6)(7.7636,0.36)(4.204,0.12)(7.796,-0.12)(7.7636,-0.36)(4.303,-0.6)(7.59197,-0.84)(4.56,-1.08)(7.2237,-1.32)(5.103,-1.56)(6,-1.8)
\psline[linestyle=dotted, linewidth=0.4pt](4,1.8)(8,1.8)
\psline[linestyle=dotted, linewidth=0.4pt](4,1.56)(8,1.56)
\psline[linestyle=dotted, linewidth=0.4pt](4,1.32)(8,1.32)
\psline[linestyle=dotted, linewidth=0.4pt](4,1.08)(8,1.08)
\psline[linestyle=dotted, linewidth=0.4pt](4,0.84)(8,0.84)
\psline[linestyle=dotted, linewidth=0.4pt](4,0.6)(8,0.6)
\psline[linestyle=dotted, linewidth=0.4pt](4,0.36)(8,0.36)
\psline[linestyle=dotted, linewidth=0.4pt](4,0.12)(8,0.12)
\psline[linestyle=dotted, linewidth=0.4pt](4,-1.8)(8,-1.8)
\psline[linestyle=dotted, linewidth=0.4pt](4,-1.56)(8,-1.56)
\psline[linestyle=dotted, linewidth=0.4pt](4,-1.32)(8,-1.32)
\psline[linestyle=dotted, linewidth=0.4pt](4,-1.08)(8,-1.08)
\psline[linestyle=dotted, linewidth=0.4pt](4,-0.84)(8,-0.84)
\psline[linestyle=dotted, linewidth=0.4pt](4,-0.6)(8,-0.6)
\psline[linestyle=dotted, linewidth=0.4pt](4,-0.36)(8,-0.36)
\psline[linestyle=dotted, linewidth=0.4pt](4,-0.12)(8,-0.12)

\rput(8.4, 1.8){\tiny $1$}
\rput(8.4,0.36){\tiny $k$}
\rput(8.4,0.12){\tiny $k+1$}
\rput(8.4, -1.8){\tiny $n+1$}
\end{pspicture}}
\caption{Picture for the proof of Lemma \ref{lem:topcenter}.}
\label{fig:prooftop}
\end{center}
\end{figure}
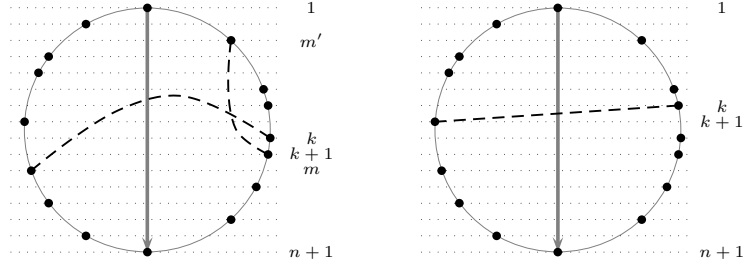
\begin{proof}
$1.$ Thanks to Remark \ref{rmq:inversec}, we can assume that $k, k+1\in R_{c}$. If a syllable $w_i$ (resp. $w_j$) has $s_k$ at its center (resp. at its top), then $w_i$ (resp. $w_j$) is in a polygon $P$ (resp. $P'$) which has $k$ as non terminal (resp. $k+1$ as non initial) index; in particular, if $P\neq P'$, then $P$ (resp. $P'$) must have an edge or a diagonal joining the point $k$ (resp. $k+1$) to a point $m>k$ and $m\neq k+1$ whence $m>k+1$ (resp. $m'<k+1$ and $m'\neq k$ whence $m'<k$). If follows that the two segments $(k,m)$ and $(m',k+1)$ are crossing (see Figure \ref{fig:prooftop}), which is a contradiction. Hence $P=P'$ and $w_i=w_j=s_k$. 


$2.$ Under these assumptions, there is a polygon $P\in\Pol(x)$ having both $k$ and $k+1$ as vertices. Another polygon $Q$ containing $s_k$ would then cross the segment joining the point $k$ to the point $k+1$ since this segment cuts the plane in two half planes, one containing the points smaller than $k$ and the other containing the points bigger than $k+1$ (see Figure~\ref{fig:prooftop}). Hence $P$ is the unique polygon containing $s_k$ and its standard form has a unique occurrence of $s_k$ since it is a center. The letter $s_{k+1}$ can then only appear as a center since otherwise we would have another occurrence of $s_k$ in $m_x^c$. 
\end{proof}

\subsection{Extracting a fully commutative subword from the standard form}

\begin{definition}\label{extracting}
We define an element $w_x^c$ of $\Sn$ by extracting a subword of $m_x^c$ representing it as follows
\begin{enumerate}
\item Each syllable $w_i$ of $m_x^c$ contributes to $w_x^c$ each simple transposition occurring in it exactly once. In particular each center of a syllable must contribute since it occurs only once in its syllable.
\item If $s_i$ is a center of a syllable and $i\in L_{c}$, then the $s_i$ are contributed from the left part of the other syllables in which they occur,
\item If $s_i$ is a center of a syllable and $i\in R_{c}$, then the $s_i$ are contributed from the right part of the other syllables in which they occur,
\item If $s_i$ is not a center and $i\in L_{c}$, then the $s_i$ are contributed from the right part of the syllables in which they occur,
\item If $s_i$ is not a center and $i\in R_{c}$, then the $s_i$ are contributed from the left part of the syllables in which they occur.
\end{enumerate}
\end{definition}
Recall that $m_x^c$ is not unique in general since for each polygon, one can have adjacent commuting syllables for which one can choose the order in which they occur; but two such syllables must correspond to reflections $(i,j)$ and $(k,\ell)$ with $i<j<k<\ell$ (since their $\S$-supports are disjoint) and hence each letter occurring in one syllable must commute to any letter in the other syllable. Hence the element of $\Sn$ represented by the extracted word is not affected by the choice of the standard form, hence $w_x^c$ is well-defined (as element of $\Sn$). We will abuse notation and write $w_x^c$ for both the element of $\Sn$ and the word for it extracted from a fixed standard form as above (in Theorem~\ref{thm:bijsame} below we will show that this word is in fact an $\S$-reduced expression of a fully commutative element).

\begin{rmq}\label{rmq:inversec2}
It follows from the above rules and Remark~\ref{rmq:inversec} that for all $x\in\NC(\Sn,c)$, we have $(w_{x}^c)^{-1}=w_{x^{-1}}^{c^{-1}}$. 
\end{rmq} 

\begin{exple}\label{lesdeuxbij}
We give an example of the above described process and make the link with the bijection $\varphi$. Let $x$ be as in Figure~\ref{figure:ordre}. Denote by $c_i$ the cycle corresponding to the polygon $P_i$, $i=1,2,3,4$. The standard forms are given by $$m_{c_1}^c=(s_6 s_5 s_4 s_5 s_6) s_3,~m_{c_2}^c=(s_7 s_6 s_5 s_6 s_7) s_8,m_{c_3}^c=(s_{12}s_{11}s_{12})s_{10},~m_{c_4}^c=s_{13}s_{12}s_{13},$$

hence concatenating the words in the order given in Figure~\ref{figure:ordre} we get 
$$m_x^c=(s_{13}s_{12}s_{13})(s_{12}s_{11}s_{12})(s_{10})(s_6 s_5 s_4 s_5 s_6)(s_3)(s_7 s_6 s_5 s_6 s_7) (s_8).$$
The letters contributed to $w_x^c$ according to the rules from Definition~\ref{extracting} are written in red
$$({s_{13}}{\color{red}s_{12}s_{13}})(s_{12}{\color{red}s_{11}s_{12}})({\color{red}s_{10}})({\color{red}s_6 s_5 s_4} s_5 s_6)({\color{red}s_3})({\color{red}s_7 s_6 s_5} s_6 s_7) ({\color{red}s_8})$$

The suwbord in red is an $\S$-reduced expression of a fully commutative element, hence $w_x^c\in\F(\Sn)$ with normal form $$w_x^c= (s_6 s_5 s_4 s_3) (s_7 s_6 s_5 ) (s_8) (s_{12} s_{11} s_{10}) ( s_{13} s_{12}).$$

We have \begin{eqnarray}\label{eqn:wx}
I_{w_x^c}=\{6,7,8,12,13\},~J_{w_x^c}=\{3,5,8,10,12\}.
\end{eqnarray}
On the other hand, we have $\inv(x)=(3,7)(5,8,9)(10,13)(12,14)$ and hence $$U_{\inv(x)}^c=\{7,8,9,13,14\},~D_{\inv(x)}^c=\{3,5,8,10,12\}.$$
It follows that $I_{\varphi(x)}=\{6,7,8,12,13\},~J_{\varphi(x)}=\{3,5,8,10,12\}$; comparing with \eqref{eqn:wx} we get that $w_x^c=\varphi(x)$. 
\end{exple}

In Example~\ref{lesdeuxbij} we have that $w_x^c$ lies in $\F(\Sn)$ and that $\varphi(x)=w_x^c$. It is a general fact:

\begin{theorem}\label{thm:bijsame}
Let $x\in \NC(\Sn,c)$. Then $w_x^c$ is a reduced expression of an element of $\F(\Sn)$ and $\varphi(x)=w_x^c$.
\end{theorem}

The proof is somewhat intricate and will occupy the remainder of this section. The strategy is to first prove that $w_x^c$ is a reduced expression of a fully commutative element by checking the criterion of Lemma~\ref{lem:fullocc}. The statement that we recover the bijection $\varphi$ will be proven later. 

\begin{lemma}\label{lem:occ}
Let $x\in\NC(\Sn,c)$. For all $i=2,\dots, n$, between any two successive occurrences of $s_i$ in $w_x^c$ there is exactly one occurrence of $s_{i-1}$. Similarly, for all $i=1,\dots, n-1$, between any two successive occurrences of $s_i$ in $w_x^c$ there is exactly one occurrence of $s_{i+1}$. 
\end{lemma}
\begin{proof}
We only prove the first statement --- the proof of the second one is similar. Thanks to Remarks \ref{rmq:inversec} and \ref{rmq:inversec2}, we can assume that $i\in R_c$. The contributions of $s_i$ to $w_x^c$ must come from distinct syllables of $m_x^c$ (even from distinct polygons). Let $w_1, w_2,\dots, w_m$ be the syllables containing $s_i$ ordered such that $i<j$ if and only if $w_i$ occurs before $w_j$ in $m_c^x$ when reading the word from the left to the right. 

First assume that $s_i$ is the center of a syllable $w$ of $m_x^c$. Then by Lemma~\ref{lem:centre2} we have $w=w_1$. Since $s_i$ cannot be the center of two syllables, the syllables $w_2,\dots, w_m$ contain both $s_i$ and $s_{i-1}$. By rule $(3)$ the syllables $w_2,\dots, w_m$ contribute their $s_i$ from their right part and their contributed $s_{i-1}$ therefore lies on the left of that contributed $s_i$. It follows that there is at least one occurrence of $s_{i-1}$ between any two successive occurrences of $s_i$ and it remains to show that we cannot have two occurrences. If it was the case, the second $s_{i-1}$ would have to be contributed from a syllable $w'$ in which it is the top (otherwise $s_i$ occurs). This would force $w_1$ and $w'$ to be syllables of the same polygon. Since $w_1$ and $w'$ do not commute (as elements of $\Sn$) and their product lies in $\NC(\Sn,c)$ the syllable $w'$ must occur before $w_1$ (since $i\in R_c$), hence this unique additional possible occurrence of $s_{i-1}$ would in that case be before the first $s_i$, contradicting the assumption that it lies between two successive occurrences of $s_i$. 

Now assume that $s_i$ is not a center. It implies that $s_{i-1}$ occurs in every syllable among the $w_1,\dots, w_m$. By rule $(5)$, these syllables contribute their $s_i$ from their left part and hence their contributed $s_{i-1}$ lies on the right of this $s_i$. As before, one can have at most one additional contribution of an $s_{i-1}$ coming from a syllable $w'$ with $s_{i-1}$ at the top. By Lemma~\ref{lem:top}, $w'$ is either the first or the last syllable containing $s_{i-1}$. We claim that in all the cases, $w'$ is the first syllable, which concludes, since it then has to occur before $w_1$. Indeed, if $s_{i-1}$ is not a center then the claim follows from Lemma~\ref{lem:top}~(2a). Hence assume that $s_{i-1}$ is a center. If $i-1\in R_c$, then by Lemma~\ref{lem:topcenter}~(1) the only way to have both $i$ and $i-1$ in $R_c$ and $s_{i-1}$ being the top and the center of two syllables is that these two syllables coincide and are equal to $s_{i-1}$, hence $w'=s_{i-1}$ in that case. The claim then follows from Lemma~\ref{lem:centre2}. If $i-1\in L_c$ then by Lemma~\ref{lem:top}~(1) $w'$ is the first syllable containing $s_{i-1}$ provided $w'\neq s_{i-1}$. But if $w'=s_{i-1}$, since $i-1\in L_c$ and $i\in R_c$ by Lemma~\ref{lem:topcenter}~(2) there is a single occurrence of $s_{i-1}$ in $m_x^c$ hence $w'$ is the unique syllable containing $s_{i-1}$, in particular the first one (note that in that case $s_i$ does not occur in $m_x^c$).  
\end{proof}

\begin{corollary}\label{cor:thmpart1}
Let $x\in\NC(\Sn,c)$. Then $w_x^c$ is a reduced expression of a fully commutative element. 
\end{corollary}
\begin{proof}
Note that $s_1$ and $s_n$ can occur at most once in $w_x^c$. Hence this follows immediately from Lemmas~\ref{lem:fullocc} and \ref{lem:occ}.
\end{proof}



We are now able to prove Theorem~\ref{thm:bijsame}. The strategy is to use the characterization of the sets $I_{\varphi(x)}$ and $J_{\varphi(x)}$ given in Theorem~\ref{thm:bijecgenerales} and Lemma~\ref{lem:ij} to show that they agree with $I_{w_x^c}$ and $J_{w_x^c}$ respectively. 

\begin{proof}[Proof of Theorem~\ref{thm:bijsame}]
The fact that $w_x^c\in\F(\Sn)$ is given by Corollary~\ref{cor:thmpart1}. To show that $\varphi(x)=w_x^c$ we show that $I_{w_x^c}=I_{\varphi(x)}$ and $J_{w_x^c}=J_{\varphi(x)}$. The claim follows since the map $g:\F(\Sn)\longrightarrow \mathcal{I}(n), w\mapsto (J_w, I_w[1])$ from Subsection~\ref{sub:bijwithful} is bijective. More precisely by Theorem~\ref{thm:bijecgenerales} it suffices to show that for any $i=2,\dots, n+1$, 
$$i-1\in I_{w_x^c}\Leftrightarrow\left\{
    \begin{array}{l}
        i\in R_{c}\mbox{ and }i\mbox{ is a non initial index of a polygon of }x\mbox{, or}\\
        i\in L_{c}\mbox{ and }i\mbox{ is a terminal index of a polygon of }x\mbox{, or}\\
        i\in L_{c}\mbox{ and }i\notin\supp(x)\mbox{ and $i$ is nested in a polygon of }x.\\ 
    \end{array}\right. 
    $$
    
and that for any $i=1,\dots, n$, 
$$i\in J_{w_x^c}\Leftrightarrow\left\{ 
    \begin{array}{l}
        i\in R_{c}\mbox{ and }s_i\mbox{ is a center of }m_x^c\mbox{, or}\\
        i\in L_{c}\mbox{ and }i\mbox{ is an initial index of a polygon of }x\mbox{, or}\\
        i\in L_{c}\mbox{ and }i\notin\supp(x)\mbox{ and }i\mbox{ is nested in a polygon of }x.\\ 
    \end{array}\right.
    $$

We only show the equivalence involving the set $J_{w_x^c}$ --- the proof of the other one is similar. Assume that $i$ satisfies one of the three conditions on the right hand side. We then show that $i\in J_{w_x^c}$. Recall from Lemma~\ref{lem:ij} that $i\in J_{w_x^c}$ in and only if there is no occurrence of $s_{i-1}$ in $w_x^c$ after the last occurrence of $s_i$. 

If $i\in R_c$ and $s_i$ is a center of $m_x^c$, then the first occurrence of $s_i$ in $m_x^c$ is at the center of a syllable by Lemma~\ref{lem:centre2}. Arguing as in the proof of Lemma~\ref{lem:occ}, we see that all the other contributions of $s_i$ come from the right of the syllables and that there can be no occurrence of $s_{i-1}$ after the last occurrence of $s_i$ because it would be at the top of a syllable and such a syllable would occur before the syllable of which $s_i$ is the center.

If $i\in L_{c}$ and $i$ is a minimal index of a polygon, then $s_i$ is a center of a syllable of $m_x^c$, hence the last occurrence of $s_i$ in $m_x^c$ is as a center (Lemma \ref{lem:centre2}). Moreover, by assumption there is no syllable with $s_{i-1}$ at its top since $i$ is the minimal index of its polygon. Hence there is no $s_{i-1}$ after the last occurrence of $s_i$ in $w_x^c$. 

If $i\in L_{c}$, $i\notin\supp(x)$ and $i$ is nested in a polygon of $x$, then by assumption there is no syllable with $s_{i-1}$ at its top (otherwise $i$ would be terminal implying $i\in\supp(x)$) and $s_i$ is contributed from the right of any syllable in which it occurs by rule~(4), implying again that there is no $s_{i-1}$ after the last occurrence of $s_i$.

In all cases we showed that there is no occurrence of $s_{i-1}$ after the last occurrence of $s_i$ in $w_x^c$. By Lemma~\ref{lem:ij} it implies that $i\in J_{w_x^c}$.  

We now assume that the condition on the right hand side is not satisfied. We first treat the case where $i=1$. In that case, the last condition is never satisfied, and if the first two are also not satisfied, then in particular $s_i$ cannot be a center of $m_x^c$. But if it occurs in $m_x^c$, it is necessary as a center since $i=1$, hence $s_i$ does not occur in $m_x^c$, hence also not in $w_x^c$. 

Now assume that $i\neq 1$. In that case the conditions $\{i\in R_c\}$ and $\{i\in L_c\}$ are disjoint. Firstly suppose that $i\in R_c$ and that $s_i$ is not a center of $m_x^c$. It follows that if $s_i$ occurs in a syllable, then $s_{i-1}$ also occurs, and since by rule~(5) the $s_i$ of any syllable is contributed from the left, there is always an $s_{i-1}$ at its right which is contributed. Therefore there is an occurrence of $s_{i-1}$ after the last occurrence of $s_i$, implying that $i\notin J_{w_x^c}$. 

Now assume that $i\in L_c$ and that $i$ is not initial and that either $i\in\supp(x)$ or $i$ is not nested in a polygon of $x$. In particular, if $s_i$ occurs in $m_x^c$ then $s_i$ is either a center or $i$ is a terminal index. If it is a center, then the last occurrence of $s_i$ in $m_x^c$ is as a center, but since $i$ is not initial there is a syllable with $s_{i-1}$ at its top which appears after the syllable with $s_i$ at its center (because $i\in L_c$). If $s_i$ is not a center then $i$ is terminal. It implies that there is a syllable with $s_{i-1}$ at its top. By Lemma~\ref{lem:top}, it must be the last syllable $w$ containing $s_{i-1}$. If $s_i$ occurs in $m_x^c$ it then occurs exactly in any other syllable distinct from $w$ which contains $s_{i-1}$ and those occur before $w$. Hence there is again an $s_{i-1}$ appearing after the last $s_i$ implying $i\notin J_{w_x^c}$. 

\end{proof}

\subsection{Triangularity}

Recall from Proposition~\ref{prop:fullya} that given $w\in\F(\Sn)$, the number of occurrences of $s_i$ is constant on $\S$-reduced expressions of $w$ and denoted by $n_i(w)$. For $x\in\NC(\Sn, c)$ also recall from Definition~\ref{def:vertical} the notation $(x_i^c)_{i=1}^n$ for the vertical vector of $x$, where $x_i^c$ is the number of occurrences of $s_i$ in $m_x^c$. It is easy to derive $n_i(w_x^c)$ from $x_i^c$: 

\begin{lemma}\label{lem:stdvert}
Let $x\in\NC(\Sn,c)$. Let $i\in[n]$. \begin{enumerate}
\item If $s_i$ is a center of $m_x^c$, then $2n_i(w_x^c)-1=x_i^c$,
\item If $s_i$ is not a center of $m_x^c$, then $2 n_i(w_x^c)= x_i^c$.
\end{enumerate}
\end{lemma}
\begin{proof}
This follows immediately from the definition of $w_x^c$: any syllable contributes any letter from $\S$ occurring in it exactly once. A letter occurs in a syllable at most twice, and it occurs once if and only if it is a center. A letter can be the center of at most one syllable (Lemma~\ref{std}).
\end{proof}

For $x\in\F(\Sn)$, we denote by $\Subf(x)$ the set of fully commutative elements having a reduced expression which is a subword of $m_x^c$. By Corollary~\ref{cor:thmpart1} we have $w_x^c\in\Subf(x)$.

\begin{exple}
Let $c=(1,3,4,2)$ and $x=(1,3)(2,4)$ as in Example~\ref{explecc}. We have $m_x^c=(s_2 s_1 s_2)(s_3 s_2 s_3)$ and $w_x^c=s_2 s_1 s_3 s_2=\varphi(x)$. We have $$\Subf(x)=\{e, s_1, s_2, s_3, s_2s_1, s_1s_2, s_2s_3, s_3s_2, s_1s_3, s_1s_2s_3, s_2s_1s_3, s_1s_3s_2, w_x^c \}.$$
Note that in that case, $m_x^c$ is not a reduced expression of $x=s_2s_1s_3s_2$. Let $w=s_1s_2s_3\in \Subf(x)$. Setting $y=(1,3,4)$ we have $\varphi(y)=w$ (see Figure~\ref{fig:thebijection}) and $m_y^c=s_2 s_1 s_2 s_3$ but $y$ is not below $x$ for Bruhat order. We have $$(y_i^c)_{i=1}^3=(1,2,1) \leq (1,3,2)=(x_i^c)_{i=1}^3$$ hence $y \leqv x$. We will show in this section that for $w\in \Subf(x)$ and $y\in\NC(\Sn,c)$ such that $\varphi(y)=w$, we have $y \leqv x$. In case $c=c_{\mathrm{lin}}$ we have that $\leqv$ coincides with the restriction of the Bruhat on $\Sn$ to $\NC(\Sn,c)$ (see~\cite{GobWil}), but the above example shows that in general it fails, and $\leqv$ is precisely the order to use as a replacement for Bruhat order in the general case to get the same poset structure as for $c_{\mathrm{lin}}$, and we will see in this section that $\leqv$ is also the right order to consider to have triangular base changes $\tl$. 
\end{exple}

\begin{lemma}\label{lem:wxunique}
Let $x\in\NC(\Sn,c)$. Then $w_x^c$ is the unique subword of $m_x^c$ which is a reduced expression of the fully commutative element $\varphi(x)$. 
\end{lemma}
\begin{proof}
Assume that there is a subword $w$ of $m_x^c$ which is an $\mathcal{S}$-reduced expression of $\varphi(x)$. Then any syllable of $m_x^c$ must contribute to that subword each letter from its support exactly once. Indeed, if not, it means that there is a syllable somewhere contributing two instances of a reflection $s_i$ in its support (or no instances of $s_i$, but in that case there must be another syllable which contributes two $s_i$ to balance the missing one), which ends in a word which cannot be a reduced expression of a  fully commutative element since there cannot be an occurrence of $s_{i+1}$ between these two $s_i$ (see Lemma~\ref{lem:fullocc}).

Now assume that the subword $w$ is distinct from the subword $w_x^c$. It means that there exists at least one syllable $w'$ of $m_x^c$ with a distinct contribution to $w$ and $w_x^c$, that is, one reflection $s_i$ appearing in $w'$ which is contributed from the left of the syllable in one case and from the right in the other case. Hence in one case, the contribution of $s_i$ from $w'$ is before the contribution of $s_{i-1}$ from $w'$ and after in the other case. Since by the first part of the proof there must be the same number of contributions of $s_i$ and $s_{i-1}$ coming from syllables lying at the left of $w'$ in both $w$ and $w_x^c$, it implies that one word among $w$ and $w_x^c$ has no occurrence of $s_i$ before the first $s_{i-1}$ while the second one has, hence that they cannot represent the same fully commutative element (Lemma~\ref{lem:ij}). It follows that $w_x^c$ is the unique subword of $m_x^c$ which is a reduced expression of $\varphi(x)$. 
\end{proof}

\begin{lemma}\label{lem:2sous}
Let $s_{i_1} s_{i_2}\cdots s_{i_k}$ be an $\S$-word. By Proposition~\ref{prop:jones}~(3), there is a unique pair $(w,m)\in\F(\Sn)\times\mathbb{Z}_{\geq 0}$ such that $$b_{i_1} b_{i_2}\cdots b_{i_k}=(v+v^{-1})^m b_w\in\tl.$$
Then there is at least one subword of $s_{i_1} s_{i_2}\cdots s_{i_k}$ which is a reduced expression of $w$. If $m>0$, there are at least two distinct subwords which are reduced expressions of $w$. 
\end{lemma}
\begin{proof}
This is a consequence of the fact that if $s_{i_1}\cdots s_{i_k}$ is not fully commutative (which is equivalent to saying that $m>0$), then one can apply either the relation $b_i^2=(v+v^{-1})b_i$ or the relation $b_i b_{i\pm 1} b_i=b_i$ in $b_{i_1} b_{i_2}\cdots b_{i_k}$ (possibly after having applied commutation relations); but since in any of these two relations, there are two $b_i$ which reduce to a single $b_i$, after applying successive such relations and possibly commutation relations, the resulting $w$ must have an $\mathcal{S}$-reduced expression which is a subword of the original word. But then it will be a subword in at least two different ways since there are two possible choices for the $s_i$. 
\end{proof}
\begin{lemma}\label{lem:stdvert2}
Let $x\in\NC(\Sn,c)$, $w\in\Subf(x)$. Let $i\in[n]$. 
\begin{enumerate}
\item If $s_i$ is a center of $m_x^c$, then $2n_i(w)-1\leq x_i^c$,
\item If $s_i$ is not a center of $m_x^c$, then $2 n_i(w)\leq x_i^c$. 
\end{enumerate}
\end{lemma}
\begin{proof}
Each syllable of $m_x^c$ containing $s_i$ can contribute at most one of the instances of $s_i$ to a subword which is a reduced expression of $w$ (otherwise we have two $s_i$ with no $s_{i+1}$ in between, contradicting Lemma~\ref{lem:fullocc}). Depending on whether $s_i$ is the center of a syllable of $m_x^c$, we get the claim (keeping in mind that a letter is the center of at most one syllable). 
\end{proof}
Note that by Lemma~\ref{lem:stdvert} we have equality in the above lemma when $w=w_x^c$. In general there are several $w\in\Subf(x)$ for which one has equality for all $i$ (the vector $(n_i(w))_{i=1}^n$ does not characterize $w$ in general), for instance this always happens if $x$ is a non-simple transposition. 

The following result is the main ingredient for proving the triangularity of the base-change between the diagram and Zinno bases of $\tl$.

\begin{theorem}\label{thm:zinnogen}
Let $x, y\in\NC(\Sn,c)$ such that $w:=\varphi(y)\in\Subf(x)$. Then $y \leqv x$. 
\end{theorem}
\begin{proof}
By definition of $\leqv$ we have to show that $y_i^c\leq x_i^c$ for all $i\in[n]$. If $s_i$ is a center of $m_y^c$ then $2 n_i(w)-1= y_i^c$ by Lemma~\ref{lem:stdvert}. In that case, Lemma~\ref{lem:stdvert2} implies that $y_i^c\leq x_i^c$. If $s_i$ is not a center of $m_{y}^c$, there would be a problem if $s_i$ was a center of $m_x^c$ and $x_i^{c}=2 n_i(w)-1$ since by Lemmas~\ref{lem:stdvert} and \ref{lem:stdvert2} we would have $y_i^{c}=2 n_i(w)>x_i^{c}$. Hence it suffices to show that if $s_i$ is a center of $m_x^c$ and $2 n_i(w)-1=x_i^{c}$, then $s_i$ is also a center of $m_{y}^c$. 

We therefore assume that $s_i$ is a center of $m_x^c$ and that $2 n_i(w)-1=x_i^{c}$. We abuse notation and will not distinguish between $w$ and a subword of $m_x^c$ which is an $\mathcal{S}$-reduced expression of $w$. By Lemma \ref{lem:stdvert} we have $x_i^{c}=2 n_i(\varphi(x))-1$, hence there is the same number of contributions of $s_i$ in $m_x^c$ to $w$ as to the subword $w_x^c$. As a consequence, arguing as in the proof of Lemma~\ref{lem:wxunique} we see that each syllable of $m_x^c$ which contains $s_i$ must contribute exactly one $s_i$ to $w$. 

If $i\in R_{c}$, then $s_i$ first occurs in $m_x^c$ as a center (Lemma~\ref{lem:centre2}) of a syllable $w'$ of a polygon $P\in\Pol(x)$; this center must therefore be contributed to both $w$ and $w_x^c$. If there is a syllable in $m_x^c$ with $s_{i-1}$ at its top, then that syllable must be in $P$ and hence it appears before the first $s_i$ (because $i\in R_c$). Any other syllable containing $s_{i-1}$ must also contain $s_i$ (and hence appears after $w'$) and these are the only other syllables in which $s_i$ appears. In particular, the $s_i$ appearing at the right of any such syllable must contribute to $w$, otherwise there would be at some place in $w$ two successive occurrences of $s_i$ with no occurrence of $s_{i-1}$ between them. In particular, the last $s_i$ is contributed from the right of a syllable and there is no $s_{i-1}$ on its right, implying $i\in J_w$. Since $i\in R_{c}$, one then has that $s_i$ is a center of $m_y^c$ by the characterization of Theorem~\ref{thm:bijecgenerales}. 

Now if $i\in L_{c}$, then thanks to Remark \ref{rmq:inversec} one gets that there is no contribution of $s_{i-1}$ before the first occurrence of $s_i$ in $m_x^c$, implying that $i-1\notin I_w$. Since $i\in L_{c}$, by Theorem~\ref{thm:bijecgenerales} again this means that one of the following two conditions is satisfied: either $i\in\supp(x)$ with $i$ not terminal or $i$ is not nested in any polygon and also not terminal. Since we in addition know that $s_i$ appears in $m_x^c$, $s_i$ must be contained in at least one polygon, hence must either be a center or nested in some polygon. Hence the two conditions force $s_i$ to be a center of $m_x^c$: the first condition says exactly that $s_i$ is a center. If the second condition is satisfied this forces $s_i$ to be a center since it has to be either a center or nested. 
\end{proof}

Recall that for $x\in\NC(\Sn,c)$, we denote by $R_x^c$ the image of $i_c(x)$ in $\tl$. The main result now follows easily

\begin{theorem}
Let $x\in\NC(\Sn,c)$. Then $$R_x^c=\sum\limits_{\substack{y\in\NC(\Sn,c), \\ y\leqv x}} c_{x,y} b_{\varphi(y)},$$
where $c_{x,y}\in\mathbb{Z}[v,v^{-1}]$ and $c_{x,x}$ is invertible. 
\end{theorem}
\begin{proof}
Recall that the standard form $m_x^c$ is an $\S$-word representing the element $x$ and that it is obtained by concatenating various words $m_{c_j}^{c}$ which are standard forms of the cycles $c_j$ corresponding to the polygons of $x$. The words $m_{c_j}^{c}$ are obtained from specific $\mathcal{T}$-reduced expressions of $c_j$ where we replaced each reflection by an $\mathcal{S}$-reduced expression of it called a syllable. As a consequence, if $m_x^c=w_1 w_2\cdots w_m$ where $w_1,\dots, w_m$ are the syllables, then the corresponding simple element $i_c(x)$ of the dual braid monoid is equal to the product of the atoms $i_{c}(t_1)\cdots i_{c}(t_m)$ where $t_1,\dots, t_m$ are the reflections of which $w_1,\dots, w_m$ are $\mathcal{S}$-reduced expressions. But recall that the syllable $w_\ell$ corresponding to the reflection $t_\ell=(i, k+1)$, $i<k+1$ is equal to the word $w_\ell =s_k s_{k-1}\cdots s_i\cdots s_{k-1} s_k$. Thanks to Lemma \ref{lem:atome}, the atom $i_{c}(t_\ell)\in B$ has an expression of the form $\mathbf{s}_k^{\varepsilon_k}\mathbf{s}_{k-1}^{\varepsilon_{k-1}}\cdots \mathbf{s}_i\cdots \mathbf{s}_{k-1}^{-\varepsilon_{k-1}}\mathbf{s}_k^{-\varepsilon_{k}}$ where $\varepsilon_j=\pm 1$ for all $j\in[i+1,k]$. It implies that $i_{c}(x)$ is represented by a word obtained from $m_x^c$ by replacing each of the $s_i$ by $\mathbf{s}_i$ or $\mathbf{s}_i^{-1}$. Now in the Temperley-Lieb algebra, $\mathbf{s}_i$ is mapped to $v^{-1}-b_i$ while $\mathbf{s}_i^{-1}$ is mapped to $v-b_i$. Hence if we expand we get a linear combination of words of the form $b_{i_1} b_{i_2}\cdots b_{i_q}$ where $s_{i_1} s_{i_2}\cdots s_{i_q}$ is a subword of $m_x^c$. Hence putting together Lemmas \ref{lem:wxunique} and \ref{lem:2sous} and Theorem \ref{thm:zinnogen}, we get the result.
\end{proof}

\begin{rmq}
In \cite[Theorem 8.16]{DG} it is shown that $(-1)^{\ls(\varphi(y))}c_{x,y}\in\mathbb{Z}_{\geq 0} [v, v^{-1}]$, but this relies on positivity results in Hecke algebras whose proofs require categorification. In general we do not have a closed formula for $c_{x,y}$. Partial results in this direction in case $c=c_{\mathrm{lin}}$ can be found in \cite[Section 5]{Gob}.  
\end{rmq}

\end{document}